\documentclass[preprint,12pt]{elsarticle}
\usepackage{amsmath,amssymb,amsthm}
\usepackage{graphicx}
\usepackage{subfigure}

\newtheorem{thm}{Theorem}[section]

 \newtheorem{cor}{Corollary}[section]
 \newtheorem{lem}{Lemma}[section]
 \newtheorem{prop}{Proposition}[section]
 \newtheorem{defn}{Definition}[section]
\newtheorem{rem}{Remark}[section]

\def\d{\partial}

\def\tilde{\widetilde}
\def\hat{\widehat}

\newcommand\R{\mathbb{R}}

\newcommand\Z{\mathbb{Z}}

\renewcommand{\div}{\mbox{\rm div}\;\!}

\def\cQ{{\mathcal Q}}
\def\cS{{\mathcal S}}

\usepackage{amssymb}

\journal{Journal of Differential Equations}

\begin{document}

\begin{frontmatter}



\title{Global existence and analyticity of $L^p$ solutions to the compressible fluid model of Korteweg
type}

\author[label1]{Zihao Song}\cortext[cor]{Corresponding author}\ead{szh1995@nuaa.edu.cn}
 \author[label1]{Jiang Xu\corref{cor}}\ead{jiangxu\underline{\ }79@nuaa.edu.cn}

 \address[label1]{School of Mathematics, Nanjing
University of Aeronautics and Astronautics, \\ Jiangsu Province, Nanjing 211106, China}

\begin{abstract}
We are concerned with a system of equations in $\mathbb{R}^{d}$ governing the evolution of isothermal, viscous and compressible
fluids of Korteweg type, that can be used as a phase transition model. In the case of zero sound speed $P'(\rho^{\ast})=0$,
it is found that the linearized system admits the \textit{purely} parabolic structure, which enables us to establish the global-in-time existence
and Gevrey analyticity of strong solutions in hybrid Besov spaces of $L^p$-type. Precisely, if the full viscosity coefficient and capillary coefficient satisfy $\bar{\nu}^2\geq4\bar{\kappa}$, then the acoustic waves are not available in compressible fluids. Consequently, the prior $L^2$ boundedness on the low frequencies of density and velocity could be improved to the general $L^p$ version with $1\leq p< d$. The proof mainly relies on new nonlinear Besov (-Gevrey) estimates for product and composition of functions.
\end{abstract}


\end{frontmatter}


\section{Introduction}\setcounter{equation}{0}
The Korteweg system aims to study the dynamics of a liquid-vapor mixture in the Diffuse Interface (DI) approach, where the phase changes are described through
the variations of the density. The theory formulation was first introduced by Van der Waals \cite{V}, later by Korteweg \cite{K} more than a century ago.  Unfortunately the
basic model provides an infinite number of solutions and in order to select the physically relevant solutions, following from Van der Waals and Korteweg, a capillary term related
to surface tension is added to classical compressible fluid equations, which penalizes the high variations of the density. The rigorous derivation of the corresponding
equations that we shall name the compressible Navier-Stokes-Korteweg system is due to
Dunn and Serrin \cite{DS}. In the barotropic case, it
reads:
\begin{equation}
\left\{
\begin{array}{l}\partial_{t}\rho+\mathrm{div}(\rho u)=0,\\ [1mm]
 \partial_{t}(\rho u)+\mathrm{div}(\rho u \otimes u)+\nabla P(\rho)=\mathcal{A}u+\mathrm{div}K(\rho).\\[1mm]
 \end{array} \right.\label{1.1}
\end{equation}
Here, $\rho=\rho(t,x)\in \mathbb{R}_{+}$ and $u=u(t,x)\in \mathbb{R}^{d}(d\geq3)$ are the unknown functions on $[0,+\infty)\times \mathbb{R}^{d}$, which stand for the density
and velocity field of a fluid, respectively. We neglect the thermal fluctuation so that the pressure $P=P(\rho)$ reduces to a function of $\rho$ only. The notation
$\mathcal{A}u$ is given by $\mathcal{A}u=\mathrm{div}(2\mu D(u))+\nabla(\lambda\mathrm{div}u)$, where the Lam\'{e} coefficients $\lambda$ and $\mu$ (the \textit{bulk} and \textit{shear viscosities}) are density-dependent functions, respectively. In order to ensure the uniform ellipticity of $\mathcal{A}u$,
they are assumed to satisfy
$$\lambda>0, \nu\triangleq\lambda+2\mu>0.$$
Set $D(u)\triangleq\frac{1}{2}(\nabla u+{}^T\!\nabla u)$, which stands for  the  deformation tensor. $\nabla$ and $\div$ are the gradient and divergence operators with respect
to the space variable. In the following, the Korteweg tensor is given by (see \cite{BDDJ})
$$\mathrm{div}K(\rho)=\nabla(\rho\kappa(\rho)\Delta\rho+\frac{1}{2}(\kappa(\rho)+\rho\kappa'(\rho))|\nabla\rho|^{2})-\mathrm{div}(\kappa(\rho)\nabla\rho\otimes\nabla\rho),$$
where $\nabla\rho\otimes\nabla\rho$ stands for the tensor product $(\partial_j\rho\partial_k\rho)_{jk}$. The capillarity coefficient $\kappa>0$
may depend on $\rho$ in general.

The initial condition of System \eqref{1.1} is prescribed by
\begin{equation}\label{Eq:1.2}
\left(\varrho,u\right)|_{t=0}=\left(\varrho _{0}(x),u_{0}(x)\right),\  x\in \R^{d}.
\end{equation}
In this paper, we investigate the Cauchy problem  \eqref{1.1}-\eqref{Eq:1.2}, where initial date tend to a constant equilibrium $(\rho^{\ast},0)$ with $\rho^{\ast}>0$.

If the capillarity coefficient $\kappa\equiv 0$, then System \eqref{1.1} reduces to the usual Navier-Stokes system of compressible fluids.
As we known, so far there are so many excellent literatures on the solutions to compressible Navier-Stokes fluids in different settings. In the present paper, allow us to pay more attention on the
Korteweg system \eqref{1.1} and investigate the dissipation effect of Korteweg
tensor. The existence of strong solutions for \eqref{1.1} was known since the works by Hattori and Li \cite{HL,HL1}. Global solutions are only obtained
for initial data \eqref{Eq:1.2} close enough to a stable equilibrium $(\rho^{\ast},0)$ with convex pressure profiles. In \cite{DD}, the well-posedness of
the Cauchy problem \eqref{1.1}-\eqref{Eq:1.2} is studied in so-called critical Besov spaces. That
is, in spaces which are invariant by the scaling of Korteweg system. Precisely, one can easily verify that if $(\rho,u)$ solves \eqref{1.1}, then does
$(\rho_{\lambda},u_{\lambda})$, where
\begin{equation}\label{scale}\rho_{\lambda}=\rho(\lambda^{2}t,\lambda x), u_{\lambda}=\lambda u(\lambda^{2}t,\lambda x),\end{equation}
provided that the pressure laws $P$ have been changed to $\lambda^{2}P$. Recall that such an approach is now classical, which has been employed in the incompressible
Navier-Stokes equations (see for example \cite{C,CP,CJY,FK}) and compressible
barotropic compressible Navier-Stokes equations (see \cite{D,CD,CMZ,DH,DX,H2,XX,X}). Danchin and Desjardins \cite{DD} first constructed global existence result for initial data close
enough to equilibrium $(\rho^{\ast},0)$ with the assumption
$P'(\rho^{\ast})>0$. Also, they gave a local-in-time existence result for initial densities bounded away from zero which does not require any stability assumption on the
pressure law. Charve, Danchin and the second author \cite{CDX} established the global well-posedness result in more general $L^p$ critical spaces. More importantly, they investigated the
Gevrey analyticity of \eqref{1.1}-\eqref{Eq:1.2}, which should be the first analyticity effort on the compressible fluids. Chikami and Kobayashi \cite{CK} investigate the
global well-posedness and optimal time-decay estimates in the
$L^2$-Besov spaces of critical regularity in higher dimensions ($d\geq3$). For the decay estimates in the general $L^p$ framework, the reader is referred to the recent work due to
Kawashima, Shibata and the second author \cite{KSX1}.
Murata and Shibata \cite{MS} addressed a totally different statement on the global existence and decay estimates to \eqref{1.1}-\eqref{Eq:1.2} in Besov spaces, where the
maximal $L^p$-$L^q$ regularity was mainly employed. Subsequently, Watanabe \cite{W} improved those results in \cite{MS} such that the
$L^2$ solution framework is admissible. The existence of global strong solution with large initial data for the Korteweg system remains an open problem in dimensions $d\geq2$. In that issue,
Haspot \cite{H3} gave a positive answer by considering shallow-water coefficients $\mu(\rho)=2\mu \rho$  and $\lambda(\rho)=0$ and quantum case
$\kappa(\rho)=\kappa/\rho$. Zhai and Li \cite{ZL} addressed well-posedness results without the smallness condition of the vertical component of the incompressible part. In  \cite{YYW}, Yu, Yang and Wu established the global-in-time existence for the pressure satisfying $P(\rho)=A\rho^\gamma$ allowing the initial data arbitrary large in $L^2$ space. More recently, Watanabe \cite{W2} proved the two dimensional system with sufficient large volume viscosity admited a global strong solution under large initial velocity.

On the other hand, owing to the fact that the Korteweg system was deduced by using Van der Waals potential, there exists  the non-monotone pressure due to the phase transition, see for example Kobayashi and Tsuda's work \cite{KT} for details.
Therefore, it is also important to investigate more physical cases of
$P'(\rho^{\ast})=0$ (zero sound speed) and $P'(\rho^{\ast})<0$ in the mathematical
analysis of
\eqref{1.1}-\eqref{Eq:1.2} where in those cases, the pressure term couldn't offer any dissipation. Danchin and Desjardins \cite{DD} focused on the critical case $P'(\rho^{\ast})=0$, and provided a Fourier study of the corresponding linearized system to observe different behaviors whether the quantity $\bar{\nu}^2-4\bar{\kappa}$ (see below) is positive, negative and zero. In all cases the parabolic smoothing in all frequency spaces are expected. In contrast to the compressible Navier-Stokes equations, the Korteweg system admits a
striking property that is still linearly stable even when the sound speed is zero. Furthermore, it was shown that the \eqref{1.1}-\eqref{Eq:1.2} was locally well-posed in more general scaling invariant Besov spaces of $L^p$ type. Kotschote \cite{K2} considered the initial
boundary value problem of \eqref{1.1} in bounded domain and proved the local existence and uniqueness of strong solutions without assuming increasing pressure laws. Tang and Gao \cite{TG} proved the stability of weak solution in the periodic domain $\mathbb{T}^d (d=2,3)$ under a non-monotone pressure law.
Chikami and Kobayashi \cite{CK} established global existence and decay of strong solution in the $L^2$ critical Besov spaces and their results were later generalized to $L^2$-$L^p$ framework for $d\geq4$ by Zhang in \cite{Zhang}. Moreover, the existence
of a global strong solution for the Navier-Stokes-Korteweg system under the condition
$P'(\rho^{\ast})=0$ is also investigated by Kobayashi and Tsuda \cite{KT} and Kobayashi and Murata \cite{KM} in the classical Sobolev spaces.
Huang, Hong and Shi \cite{HHS} studied more general pressure law including Van der Waals equation of state and showed the local-in-time existence of smooth solutions to the Cauchy problem of \eqref{1.1}-\eqref{Eq:1.2}. Furthermore, the global-in-time existence of smooth solutions for the periodic problem is also established.

Following from \cite{DD}, \textit{a question remains whether one can obtain the unique global solution in the critical Besov spaces of $L^p$ type}. Recently, it is observed that the linearized system of \eqref{1.1}-\eqref{Eq:1.2} admits a \textit{purely} parabolic structure, where acoustic waves are not available in particular as $\bar{\nu}^2\geq4\bar{\kappa}$ (see Section 2 for details). Consequently, the new observation allows to establish the $L^p$ version of global well-posedness in both high frequencies and low frequencies of density and velocity.

\subsection{Momentum formulation and main results}
Owing to the technical reason (taking care of conservation form in fact), we would like to rewrite \eqref{1.1} in terms of momentum.
Let the reference density $\rho^{*}>0$ such that $P'(\rho^{*})=0$. Denote by $a=(\rho-\rho^{*})/\rho^{*}$ the density fluctuation and by $m=\rho u/\rho^{*}$
the scaled momentum. Also, we introduce the scaled viscosity coefficients $\bar{\mu}=\frac{\mu(\rho^{*})}{\rho^{*}}$, $\bar{\lambda}=\frac{\lambda(\rho^{*})}{\rho^{*}}$
and the scaled capillarity coefficient $\bar{\kappa}=\kappa(\rho^{*})\rho^{*}$. A simple calculation leads us to the following perturbation problem:
\begin{equation}
\left\{
\begin{array}{l}\partial_{t}a+\mathrm{div}m=0,\\ [1mm]
 \partial_{t}m-\bar{\mathcal{A}}m-\bar{\kappa}\nabla\Delta a= g(a,m),\\[1mm]
(a,m)|_{t=0}=(a_{0},m_{0}),\\[1mm]
 \end{array} \right.\label{linearized}
\end{equation}
with
$$a_{0}\triangleq (\rho_{0}-\rho^{*})/\rho^{*}, \quad  \ m_{0}\triangleq \rho_{0} u_{0}/\rho^{*} \quad \mbox{and} \quad \bar{\mathcal{A}}m=\bar{\mu}\mathrm{div}D(m)+(\bar{\mu}+\bar{\lambda})\nabla\mathrm{div}m.$$
Setting $\bar{\nu}\triangleq2\bar{\mu}+\bar{\lambda}$ and $\check{\kappa}\triangleq\kappa(\rho^{*})+\rho^{*}\kappa'(\rho^{*})$. One can write the nonlinear terms $g=\sum^{6}_{i=1} g_{i}$ as follows:
\begin{equation}
\left\{
\begin{array}{l}g_{1}=\mathrm{div}\big((Q(a)-1)m\otimes m\big),\\ [1mm]
g_{2}=\bar{\mu}\Delta(Q(a)m)+(\bar{\mu}+\bar{\lambda})\nabla\mathrm{div}(Q(a)m),\\ [1mm]
g_{3}=2\mathrm{div}\Big(\tilde{\mu}(a)D\big((1-Q(a))m\big)\Big)+\nabla\Big(\tilde{\lambda}(a)\mathrm{div}\big((1-Q(a))m\big)\Big),\\ [1mm]
g_{4}=\nabla(aG(a)),\\ [1mm]
g_{5}=\nabla(\tilde{\kappa}_{1}(a)\Delta a),\\ [1mm]
g_{6}=\frac{\rho^{*}}{2}\nabla\Big(\big(\tilde{\kappa}_{2}(a)+\check{\kappa}\big)|\nabla a|^{2}\Big)-\mathrm{div}\Big(\big(\tilde{\kappa}_{3}(a)+\bar{\kappa}\big)\nabla
a\otimes\nabla a\Big),\\ [1mm]
 \end{array} \right.\label{nonlinear}
\end{equation}
with
\begin{equation}
\left\{
\begin{array}{l}Q(a)=\frac{a}{a+1},\\ [1mm]
G(a)=P'(\rho^{*}+\theta \rho^{*}a),\\ [1mm]
\tilde{\mu}(a)=\frac{2\mu(a\rho^{*}+\rho^{*})}{\rho^{*}}-2\bar{\mu},\\ [1mm]
\tilde{\lambda}(a)=\frac{\lambda(a\rho^{*}+\rho^{*})}{\rho^{*}}-\bar{\lambda},\\ [1mm]
\tilde{\kappa}_{1}(a)=(a\rho^{*}+\rho^{*})\kappa(a\rho^{*}+\rho^{*})-\bar{\kappa},\\ [1mm]
\tilde{\kappa}_{2}(a)=\kappa(a\rho^{*}+\rho^{*})+(a\rho^{*}+\rho^{*})\kappa'(a\rho^{*}+\rho^{*})-\check{\kappa},\\ [1mm]
\tilde{\kappa}_{3}(a)=\rho^{*}\kappa(a\rho^{*}+\rho^{*})-\bar{\kappa}.\\ [1mm]
 \end{array} \right.\label{composite1}
\end{equation}
In our analysis, those functions are to be vanish at zero and the exact value will not matter.
The main goal of our paper is to establish the global-in-time existence and Gevrey analyticity of strong solutions (\ref{linearized})-(\ref{composite}) in the case of
$P'(\rho^{\ast})=0$. For that end,
we would like to introduce the following hybrid Besov space.

\begin{defn}\label{defn1.2}
For $T>0$ and $1\leq p,q\leq\infty$, we denote by $E^{p,q}_{T}$  the space of functions such that \begin{equation}
\|(a,m)\|_{E^{p,q}_{T}}\triangleq
\|(\nabla a,m)\|_{\tilde{L}^{\infty}_{T}(\dot{B}^{\frac{d}{p}-1,\frac{d}{q}-3}_{(p,1),(q,\infty)})}+\|(\nabla
a,m)\|_{\tilde{L}^{1}_{T}(\dot{B}^{\frac{d}{p}+1,\frac{d}{q}-1}_{(p,1),(q,\infty)})},
\end{equation}
where $\dot{B}^{s,t}_{(p,1),(q,\infty)}$ are the hybrid Besov spaces (see Section 3). The time index $T$ can be omitted if
$T=\infty$.
\end{defn}
Some assumptions are labeled as follows. \begin{itemize}
\item[$(H_1)$:] $P'(\rho^{*})=0$;
\item[$(H_2)$:]$\lambda, \mu, \kappa$ and $P$ are real analytic\footnote{Those functions are assumed to be real analytic near zero in order to establish the evolution of Gevrey regularity. In fact, the usual smoothness is sufficient for the global existence of solutions.}; \\
\item[$(H_3)$:]  $(a_{0},m_{0})\in \dot B^{\frac d{p}}_{p,1}\times\dot
B^{\frac d{p}-1}_{p,1}$, besides, $(a_0^\ell,u_0^\ell)\in{{\dot
    B}^{\frac{d}{q}-2}}_{q,\infty}\times{{\dot B}^{\frac{d}{q}-3}}_{q,\infty}$ \mbox{such that}\begin{eqnarray*}
\|(\nabla a_{0}, m_{0})\|_{\dot{B}^{\frac{d}{p}-1,\frac{d}{q}-3}_{(p,1),(q,\infty)}}\ll 1.\end{eqnarray*}
\end{itemize}
Now, the main result is stated as follows.
\begin{thm}\label{thm2.1}Suppose that $\bar{\nu}^2\geq4\bar{\kappa}$. Let $1 \leq q\leq p\leq 2q$, $p<d$ and
\begin{eqnarray}\label{condtion1}\frac{1}{q}< \frac{1}{p}+\frac{2}{d}.\end{eqnarray}
If $(H_1)\sim (H_3)$ hold true, then the Cauchy problem (\ref{linearized})-(\ref{composite1}) admits a unique global-in-time solution $(a,m)$ in the space $E^{p,q}$ satisfying
\begin{equation}\label{bound}
\|(a,m)\|_{E^{p,q}_{T}}\lesssim \|(\nabla a_{0}, m_{0})\|_{\dot{B}^{\frac{d}{p}-1,\frac{d}{q}-3}_{(p,1)(q,\infty)}}
\end{equation}
for any $T>0$. Furthermore, the solution $(a,m)$ fulfills $e^{\sqrt {c_{0}t}\Lambda_1}(a,m)\in E^{p,q}$ for $q\neq1$,
where $c_{0}>0$ and $\Lambda_1$  stands for the Fourier multiplier with
symbol $|\xi|_{1}=\sum\limits_{i=1}^{d}|\xi_{i}|.$\footnote{For technical reasons, as observed before in \cite{L-cras}, it is more
convenient to use the  $\ell^1(\R^d)$ norm instead of the usual $\ell^2(\R^d)$ norm associated with $\Lambda=(-\Delta)^{1/2}$.}
\end{thm}

\begin{rem}
The global-in-time strong solutions of \eqref{1.1}-\eqref{Eq:1.2} can be constructed in $E^{p,p}$ with $1\leq p<d$, under the assumptions of $P'(\rho^{\ast})=0$
and $\bar{\nu}^2\geq4\bar{\kappa}$. As we known, the $L^2$-type boundedness on the low frequencies of density and velocity seems to be unavoidable in the well-prepared data framework due to acoustic waves, see e.g. \cite{CD,CMZ,DH,H2} for compressible Navier-Stokes equations and \cite{CDX,CK,DD} for compressible Navier-Stokes-Korteweg equations. The special structure enables us to go beyond the limit, which is the novelty of this paper. To do this, general product estimates of $L^q$-$L^p(1\leq q\leq p)$ type have been developed, see Proposition \ref{cor1.01} and Corollary \ref{cor1.1}.
\end{rem}

\begin{rem}
The research of analyticity (in space or time variables) for incompressible Navier-Stokes equations has a long history, see \cite{BBT,FT1,FT2,M} and therein references. Several years ago, Charve, Danchin and the second author \cite{CDX} first investigated the Gevrey analyticity of solutions for compressible fluids model of Korteweg type. Theorem \ref{thm2.1}, in fact, gives the counterpart in the case of $P'(\rho^{\ast})=0$, which indicates that if initial data are sufficiently small in critical $L^p$ hybrid Besov spaces, then the solution to \eqref{1.1}-\eqref{Eq:1.2} is globally in the Gevrey class and the radius of analyticity grows like $\sqrt{t}$ in time.
\end{rem}

\subsection{Motivation and strategy}
In order to understand the proof of Theorems \ref{thm2.1} well, we make a formal spectral analysis to the linearized system of \eqref{linearized}.  Denote by $\mathcal{P}=\mathrm{Id}-\nabla(-\Delta)^{-1} \mathrm{div}$ the Leray projector. Then, the linear system can be written in terms of
the divergence-free part $\mathcal{P}m$ and the compressible one $\mathcal{Q}m=(I-\mathcal{P})m$:
\begin{equation}
\left\{
\begin{array}{l}\partial_{t}a+\mathrm{div} \mathcal{Q}m=0,\\ [1mm]
 \partial_{t}\mathcal{Q}m-\bar{\nu}\Delta \mathcal{Q}m-\bar{\kappa}\nabla\Delta a=0,\\[1mm]
\partial_{t}\mathcal{P}m-\bar{\mu}\Delta \mathcal{P}m =0.\\[1mm]
 \end{array} \right.\label{spectrum}
\end{equation}

It is clearly that the incompressible part $\mathcal{P}m$ just satisfies an ordinary heat equation. Regarding the compressible part $\mathcal{Q}m$, it is convenient to
introduce $\mathcal{V}\triangleq \Lambda^{-1}\div m$.  Consequently, the new variable $(a, \mathcal{V})$ satisfies the coupling $2\times 2$ system:
\begin{equation}\label{R-E69}
\left\{\begin{array}{l}\d_ta+\Lambda \mathcal{V}=0,\\[1ex]
\d_t\mathcal{V}-\bar{\nu}\Delta \mathcal{V}-\kappa\Lambda^{3} a=0
\end{array}\right.
\end{equation}
Taking the Fourier transform with respect to $x\in \mathbb{R}^{d}$ leads to
\begin{equation}
\frac{d}{dt}\left(
              \begin{array}{c}
                \hat{a} \\
                \hat{\mathcal{V}} \\
              \end{array}
            \right)
=A(\xi)\left(
              \begin{array}{c}
                \hat{a} \\
                \hat{\mathcal{V}} \\
              \end{array}
            \right)
           \quad \mbox{with}\quad A(\xi)=\left(
                                                          \begin{array}{cc}
                                                            0 & -|\xi| \\
                                                            \kappa|\xi|^3 & -\bar{\nu}|\xi|^2 \\
                                                          \end{array}
                                                        \right),
\end{equation}
where $\xi\in \mathbb{R}^{d}$ is the Fourier variable. It is not difficult to check that
\begin{itemize}
\item[$(i)$:] If $\bar{\nu}^2\geq4\bar{\kappa}$, then $A(\xi)$ has two real eigenvalues:
$$\lambda_{\pm}=\frac{-\bar{\nu}\pm\sqrt{(\bar{\nu}^{2}-4\bar{\kappa})}}{2}|\xi|^{2};$$
\item[$(ii)$:] If $\bar{\nu}^2<4\bar{\kappa}$, then $A(\xi)$ has two complex conjugated eigenvalues:
$$\lambda_{\pm}=\frac{-\bar{\nu}\pm \mathrm{i}\sqrt{4\bar{\kappa}-\bar{\nu}^{2}}}{2}|\xi|^{2},$$
where $\mathrm{i}=\sqrt{-1}$ is the unit imaginary number.
\end{itemize}

Let us underline that \eqref{R-E69} is of the ``\textit{regularity-gain type}", according to the dissipation definition for general hyperbolic-parabolic system with dispersion formulated
in \cite{KSX2}, which indicates that the solution admits parabolic regularization. In particular, the system
is \textit{ purely} parabolic in the case $(i)$, which is our main motivation. The study for the case $(ii)$ will be shown in another forthcoming paper. Haspot \cite{H3}
constructed the existence of global strong solution to isothermal quantum Navier-Stokes equations with \textit{large} initial data, and that result
is strongly related to a specific choice in $(i)$. A natural question follows. Is it possible to produce a large solution for general viscosity
and capillarity coefficients? On the other hand, the large-time asymptotic description of solutions in case of $1\leq p<2$ is also very interesting (see \cite{HZ}). These questions are left to the future consideration.

In what follows, we sketch the strategy for the proof of Theorems \ref{thm2.1}. The uniform a priori estimates depend mainly on the energy approach in terms
of Hoff's \textit{viscous effective flux} (see \cite{HO}), which has been well developed by Haspot \cite{H2} in the critical regularity framework.
In the case $(i)$, the effective flux can be defined by
\begin{equation}\label{effective}
w\triangleq \cQ m+\alpha \nabla a,\end{equation}
where $\alpha=\frac{1}{2}(\bar{\nu}\pm\sqrt{\bar{\nu}^2-4\bar{\kappa}})$. Owing to the loss of dissipation arising from pressure, the low-frequency regularity is totally
different in contrast to the case of $P'(\rho^{\ast})>0$. To control the evolution of norm $\|a\|_{L^{1}_{T}(\dot{B}^{d/p}_{p,1})}$ at low frequencies, inspired by \cite{CK}, a lower regularity than scaling is employed, which brings us to use ``hybrid" Besov norms with different regularity indices in the low and high frequencies (chain of spaces ${{\dot
B}^{\frac{d}{q}-2}}_{q,\infty}\times{{\dot B}^{\frac{d}{q}-3}}_{q,\infty}(q\leq p)$ for low frequencies and
$\dot B^{\frac d{p}}_{p,1}\times\dot B^{\frac d{p}-1}_{p,1}$ for high frequencies) which allows us to construct solutions with a more rough initial data (compared to ${{\dot B}^{\frac{d}{q}-3}}_{q,1}$-type space in \cite{CK}). Consequently, the more general $L^q-L^p$ type prduct estimates would be established. The global-in-time existence is proved by the standard Banach fixed point theorem.

So far there are few analyticity results for compressible fluids, except for \cite{CDX}. One of major problems is how to bound the evolution of Gevrey regularity on those nonlinearity terms (for instance, product and composite). In the case of $P'(\rho^{\ast})=0$, one has to handle new nonlinear terms like $\nabla Q(a)m, DQ(a)m$ in $g_2$ and $g_3$. Consequently, a composite estimate with respect to real analytic functions is developed, where the regularity index can be relaxed as $s>d/p$ in comparison with Proposition \ref{prop5.2} (see Proposition \ref{prop4.1}).

Finally, the rest of this paper unfolds as follows: in Section 3, we briefly recall the Littlewood-Paley decomposition and Besov space theory.
Section 4 is devoted to the proof of Theorem \ref{thm2.1}. In the last section (``Appendix"), we develop those nonlinear
estimates for paraproduct, remainder and composition.

\section{Littlewood-Paley theory and Besov spaces}\label{sec:2} \setcounter{equation}{0}
Throughout the paper, $C>0$ stands for a generic ``constant". For brevity, $f\lesssim g$ means that $f\leq Cg$. It will also be understood that $\|(f,g)\|_{X}=\|f\|_{X}+\|g\|_{X}$ for all $f,g\in X$. The Fourier transform of a function $f\in\mathcal{S}$ (the Schwarz class) is denoted by
$$\widehat{f}(\xi)=\mathcal{F}[f](\xi):=\int_{\mathbb{R}^{d}}f(x)e^{-i\xi\cdot x}dx.$$
For $ 1\leq p\leq \infty$, we denote by $L^{p}=L^{p}(\mathbb{R}^{d})$ the usual Lebesgue space on $\mathbb{R}^{d}$ with the norm $\|\cdot\|_{L^{p}}$.

For convenience of reader, we would like to recall the Littlewood-Paley decomposition, Besov spaces and related analysis tools. The reader is referred to Chap. 2 and Chap. 3 of \cite{BCD} for more details. Let $\chi$ be a smooth function valued in $[0,1]$, such that $\chi$ is supported in the ball
$\mathbf{B}(0,\frac{4}{3})=\{\xi\in\mathbb{R}^{d}:|\xi|\leq\frac{4}{3}\}$. Set $\varphi(\xi)=\chi(\xi/2)-\chi(\xi)$. Then $\varphi$
is supported in the shell $\mathbf{C}(0,\frac{3}{4},\frac{8}{3})=\{\xi\in\mathbb{R}^{d}:\frac{3}{4}\leq|\xi|\leq\frac{8}{3}\}$ so that
$$\sum_{q\in\mathbb{Z}}\varphi(2^{-q}\xi)=1, \quad \forall\xi\in\mathbb{R}^{d}\backslash\{{0}\}.$$

For any tempered distribution $f\in\mathcal{S}'$, one can define the homogeneous dyadic blocks and homogeneous low-frequency cut-
off operators:
\begin{eqnarray*}
&&\dot{\Delta}_{q}f:=\varphi(2^{-q}D)f=\mathcal{F}^{-1}(\varphi(2^{-q}\xi)\mathcal{F}f), \quad q\in\mathbb{Z};
\end{eqnarray*}
\begin{eqnarray*}
&&\dot{S}_{q}f:=\chi(2^{-q}D)f=\mathcal{F}^{-1}(\chi(2^{-q}\xi)\mathcal{F}f), \quad q\in\mathbb{Z}.
\end{eqnarray*}
Furthermore, we have the formal homogeneous decomposition as follows
$$f=\sum_{q\in\mathbb{Z}}\dot{\Delta}_{q}f.$$
Also, throughout the paper, $f^{h}$ and $f^{\ell}$ represent the high frequency part and low frequency part of $f$ respectively where
$$f^{\ell}\triangleq\dot{S}_{q_{0}}f;\quad f^{h}\triangleq(1-\dot{S}_{q_{0}})f$$
with some given constant $q_{0}$.

Denote by $\mathcal{S}'_{0}:=\mathcal{S'}/\mathcal{P}$ the tempered distributions modulo polynomials $\mathcal{P}$. As we known, the homogeneous
Besov spaces can be characterised in terms of the above spectral cut-off blocks.

\subsection{{\bf Homogeneous Besov space}}
\hspace*{\fill}
\begin{defn}\label{defn2.1}
For $s\in \mathbb{R}$ and $1\leq p,r\leq \infty$, the homogeneous Besov spaces $\dot{B}^s_{p,r}$ are defined by
$$\dot{B}^s_{p,r}:=\Big\{f\in \mathcal{S}'_{0}:\|f\|_{\dot{B}^s_{p,r}}<\infty  \Big\} ,$$
where
\begin{equation*}
\|f\|_{\dot{B}^s_{p,r}}:=\Big(\sum_{q\in\mathbb{Z}}(2^{qs}\|\dot{\Delta}_qf\|_{L^{p}})^{r}\Big)^{1/r}
\end{equation*}
with the usual convention if $r=\infty$.
\end{defn}

We often use the following classical properties of Besov spaces (see \cite{BCD}):

$\bullet$ \ \emph{Scaling invariance:} For any $\sigma\in \mathbb{R}$ and $(p,r)\in
[1,\infty ]^{2}$, there exists a constant $C=C(\sigma,p,r,d)$ such that for all $\lambda >0$ and $f\in \dot{B}_{p,r}^{\sigma}$, we have
$$
C^{-1}\lambda ^{\sigma-\frac {d}{p}}\|f\|_{\dot{B}_{p,r}^{\sigma}}
\leq \|f(\lambda \,\cdot)\|_{\dot{B}_{p,r}^{\sigma}}\leq C\lambda ^{\sigma-\frac {d}{p}}\|f\|_{\dot{B}_{p,r}^{\sigma}}.
$$

$\bullet$ \ \emph{Completeness:} $\dot{B}^{\sigma}_{p,r}$ is a Banach space whenever $
\sigma<\frac{d}{p}$ or $\sigma\leq \frac{d}{p}$ and $r=1$.

$\bullet$ \ \emph{Interpolation:} The following inequality is satisfied for $1\leq p,r_{1},r_{2}, r\leq \infty, \sigma_{1}\neq \sigma_{2}$ and $\theta \in (0,1)$:
$$\|f\|_{\dot{B}_{p,r}^{\theta \sigma_{1}+(1-\theta )\sigma_{2}}}\lesssim \|f\| _{\dot{B}_{p,r_{1}}^{\sigma_{1}}}^{\theta} \|f\|_{\dot{B}_{p,r_2}^{\sigma_{2}}}^{1-\theta }$$
with $\frac{1}{r}=\frac{\theta}{r_{1}}+\frac{1-\theta}{r_{2}}$.

$\bullet$ \ \emph{Action of Fourier multipliers:} If $F$ is a smooth homogeneous of
degree $m$ function on $\mathbb{R}^{d}\backslash \{0\}$ then
$$F(D):\dot{B}_{p,r}^{\sigma}\rightarrow \dot{B}_{p,r}^{\sigma-m}.$$


The embedding properties will be used several times throughout the paper.
\begin{prop} \label{Prop2.1}
\begin{itemize}
\item For any $p\in[1,\infty]$ we have the continuous embedding
$\dot {B}^{0}_{p,1}\hookrightarrow L^{p}\hookrightarrow \dot{B}^{0}_{p,\infty}$.
\item If $\sigma\in \mathbb{R}$, $1\leq p_{1}\leq p_{2}\leq\infty$ and $1\leq r_{1}\leq r_{2}\leq\infty,$
then $\dot {B}^{\sigma}_{p_1,r_1}\hookrightarrow
\dot {B}^{\sigma-d\,(\frac{1}{p_{1}}-\frac{1}{p_{2}})}_{p_{2},r_{2}}$.
\item The space $\dot {B}^{\frac {d}{p}}_{p,1}$ is continuously embedded in the set of
bounded  continuous functions (going to zero at infinity if, additionally, $p<\infty$).
\end{itemize}
\end{prop}

Also, the following proposition will be used in this paper.
\begin{prop}\label{prop2.1}
Let $\sigma\in \mathbb{R}$ and $1\leq p, r\leq \infty$. Let $(f_{j})_{j\in \mathbb{Z}} $ be a sequence of
$L^p$ functions such that $\sum_{j\in\Z} f_j$ converges to some distribution $f$ in $\cS'_0$ and
$$\Bigl\|2^{j\sigma}\|f_j\|_{L^p(\R^d)}\Bigr\|_{\ell^r(\Z)}<\infty.
$$
If $\mathrm{Supp} \hat{f}_{j}\subset\mathcal{C}(0,2^jR_{1},2^jR_{2})$ for some $0<R_{1}<R_{2},$ then $f$ belongs to $\dot{B}^{\sigma}_{p,r}$ and there exists a constant $C$ such that
 \begin{equation*}
\|f\|_{\dot{B}^{\sigma}_{p,r}}\leq C\Bigl\|2^{j\sigma}\|f_j\|_{L^p(\R^d)}\Bigr\|_{\ell^r(\Z)}\cdotp
 \end{equation*}
 \end{prop}

Moreover, a class of mixed space-time Besov spaces are also used when studying the evolution PDEs, which were firstly
proposed by J.-Y. Chemin and N. Lerner in \cite{CL}.
\begin{defn}\label{defn2.2}
For $T>0,s\in \mathbb{R}, 1\leq r,\theta \leq \infty$, the homogeneous Chemin-Lerner spaces $\tilde{L}^\theta_{T}(\dot{B}^s_{p,r})$ are defined by
$$\tilde{L}^\theta_{T}(\dot{B}^s_{p,r}):=\Big\{f\in L^\theta(0,T;\mathcal{S}'_{0}) :\|f\|_{\tilde{L}^\theta_{T}(\dot{B}^s_{p,r})}<\infty  \Big\}, $$
where
$$\|f\|_{\tilde{L}^\theta_{T}(\dot{B}^s_{p,r})}:=\Big(\sum_{q\in \mathbb{Z}}(2^{qs}\|\dot{\Delta}_qf\|_{L^\theta_{T}(L^{p})})^{r}\Big)^{1/r} $$
with the usual convention if  $r=\infty$.
\end{defn}

The Chemin-Lerner space $\widetilde{L}^{\theta}_{T}(\dot{B}^{s}_{p,r})$ may be linked with the standard spaces $L_{T}^{\theta}(\dot{B}^{s}_{p,r})$ by means of Minkowski's inequality.
\begin{rem}\label{Rem2.1}
It holds that
$$\left\|f\right\|_{\widetilde{L}^{\theta}_{T}(\dot{B}^{s}_{p,r})}\leq\left\|f\right\|_{L^{\theta}_{T}(\dot{B}^{s}_{p,r})}\,\,\,
\mbox{if} \,\, \, r\geq\theta;\ \ \ \
\left\|f\right\|_{\widetilde{L}^{\theta}_{T}(\dot{B}^{s}_{p,r})}\geq\left\|f\right\|_{L^{\theta}_{T}(\dot{B}^{s}_{p,r})}\,\,\,
\mbox{if}\,\,\, r\leq\theta.
$$
\end{rem}

\subsection{{\bf Product estimates and composition estimates}}
\hspace*{\fill}

The product estimates in Besov spaces play a fundamental role in bounding bilinear terms in \eqref{linearized} (see \cite{BCD}). Let us first introduce the well-known Bony decomposition where for any distribution $f$ and $g$, there holds
$$fg=T_{f}g+R(f,g)+T_{g}f$$
where paraproducts and remainders are defined as follow
$$T_{f}g\triangleq\sum_{j'\in \mathbb{Z}}\dot{S}_{j'-1}f\dot{\Delta}_{j'}g,\qquad R(f,g)\triangleq\sum_{j'\in \mathbb{Z}}\tilde{\dot{\Delta}}_{j'}f\dot{\Delta}_{j'}g.$$
where $\tilde{\dot{\Delta}}_{j'}=\sum\limits_{|j'-k|\leq1}{\dot{\Delta}}_{k}$.

Based on the Bony decomposition, we immediately have the following Proposition concerns bounding bilinear terms in \eqref{linearized} (see \cite{BCD})
\begin{prop}\label{prop3.2}
Let $s>0$ and $1\leq p,\,r\leq\infty$. Then $\dot{B}^{s}_{p,r}\cap L^{\infty}$ is an algebra and
$$
\|fg\|_{\dot{B}^{s}_{p,r}}\lesssim \|f\|_{L^{\infty}}\|g\|_{\dot{B}^{s}_{p,r}}+\|g\|_{L^{\infty}}\|f\|_{\dot{B}^{s}_{p,r}}.
$$
If $s_{1},s_{2}\leq\frac{d}{p}$, $s_{1}+s_{2}>d\max\{0,\frac{2}{p}-1\}$, then
$$\|ab\|_{\dot{B}^{s_{1}+s_{2}-\frac{d}{p}}_{p,1}}\lesssim\|a\|_{\dot{B}^{s_{1}}_{p,1}}\|b\|_{\dot{B}^{s_{2}}_{p,1}}.$$
If $s_{1}\leq\frac{d}{p}$, $s_{2}<\frac{d}{p}$, $s_{1}+s_{2}\geq d\max\{0,\frac{2}{p}-1\}$, then
$$\|ab\|_{\dot{B}^{s_{1}+s_{2}-\frac{d}{p}}_{p,\infty}}\lesssim\|a\|_{\dot{B}^{s_{1}}_{p,1}}\|b\|_{\dot{B}^{s_{2}}_{p,\infty}}.$$
\end{prop}

System  \eqref{linearized}  also involves compositions of functions that
are handled according to the following estimates.

\begin{prop}\label{prop2.25}
Let $F:\mathbb{R}\rightarrow \mathbb{R}$ be smooth with $F(0)=0$. For all $1\leq p,\,r\leq\infty$ and $s>0$ we have
$F(f)\in \dot {B}^{s}_{p,r}\cap L^{\infty}$ for $f\in \dot{B}^{s}_{p,r}\cap L^{\infty}$, and
$$\|F(f)\|_{\dot B^{s}_{p,r}}\leq C\|f\|_{\dot B^{s}_{p,r}}$$
with $C>0$ depending only on $\|f\|_{L^{\infty}}$, $F'$ (and higher derivatives), $s$, $p$ and $d$.

In the case $s>-d\min(\frac {1}{p},\frac {1}{p'})$ then $f\in\dot{B}^{s}_{p,r}\cap\dot {B}^{\frac {d}{p}}_{p,1}$
implies that $F(f)\in \dot{B}^{s}_{p,r}\cap\dot {B}^{\frac {d}{p}}_{p,1}$, and
$$\|F(f)\|_{\dot B^{s}_{p,r}}\leq C\|f\|_{\dot {B}^{s}_{p,r}},$$
where $C>0$ is some constant depends on $\|f\|_{\dot B^{\frac{d}{p}}_{p,1}}$, $F, s, p$ and $d$.
\end{prop}

Next we introduce the following Proposition concerns error estimates in the Chemin-Lerner space. The readers could refer to \cite{DD} for proofs.
\begin{prop}\label{prop2.26}
Let $F:\mathbb{R}\rightarrow \mathbb{R}$ be smooth with $F(0)=0$. For all $1\leq p,\,r, \theta\leq\infty$ and $-d\min(\frac{1}{p},\frac{1}{p'})<s<\frac{d}{p}$, if $f_{i}\in\tilde{L}^{\theta}_{T}(\dot{B}^{s}_{p,r})\cap\tilde{L}^{\infty}_{T}(\dot {B}^{\frac {d}{p}}_{p,1})$ ($i=1, 2$), we have that
$$\|F(f_{1})-F(f_{2})\|_{\tilde{L}^{\theta}_{T}(\dot B^{s}_{p,r})}\leq C\|f_{1}-f_{2}\|_{\tilde{L}^{\theta}_{T}(\dot {B}^{s}_{p,r})},$$
where $C>0$ is some constant depends on $\|f_{i}\|_{\tilde{L}^{\infty}_{T}(\dot B^{\frac{d}{p}}_{p,1})}$, $F, s, p$ and $d$. Moreover, the case $s=\frac{d}{p}$ holds true when $r=1$.
\end{prop}

In addition, we also recall the classical \emph{Bernstein inequality}:
\begin{equation}\label{Eq:2.6}
\|D^{k}f\|_{L^{b}}
\leq C^{1+k} \lambda^{k+d(\frac{1}{a}-\frac{1}{b})}\|f\|_{L^{a}}
\end{equation}
that holds for all function $f$ such that $\mathrm{Supp}\,\mathcal{F}f\subset\left\{\xi\in \mathbb{R}^{d}: |\xi|\leq R\lambda\right\}$ for some $R>0$
and $\lambda>0$, if $k\in\mathbb{N}$ and $1\leq a\leq b\leq\infty$.

More generally, if we assume $f$ to satisfy $\mathrm{Supp}\,\mathcal{F}f\subset\{\xi\in \mathbb{R}^{d}:
R_{1}\lambda\leq|\xi|\leq R_{2}\lambda\}$ for some $0<R_{1}<R_{2}$ and $\lambda>0$, then for any smooth
homogeneous of degree $m$ function $A$ on $\mathbb{R}^d\setminus\{0\}$ and $1\leq a\leq\infty$, we have
(see e.g. Lemma 2.2 in \cite{BCD}):
\begin{equation}\label{Eq:2.7}
\|A(D)f\|_{L^{a}}\approx\lambda^{m}\|f\|_{L^{a}}.
\end{equation}
An obvious  consequence of (\ref{Eq:2.6}) and (\ref{Eq:2.7}) is that
$\|D^{k}f\|_{\dot{B}^{s}_{p, r}}\thickapprox \|f\|_{\dot{B}^{s+k}_{p, r}}$ for all $k\in\mathbb{N}$.

Let us recall the maximal regularity property of heat equations.
\begin{prop}\label{prop2.2}
Let $s \in \mathbb{R}$, $(p,r)\in [1,\infty]^{2}$ and $1\leq\rho_{1}\leq\rho\leq\infty.$ Let $u$ satisfies
\begin{equation}\label{3.1}
\left\{
\begin{array}{l}\partial_{t}u-\mu\Delta u=f,\\ [1mm]
u|_{t=0}=u_{0}(x).\\[1mm]
 \end{array} \right.
\end{equation}
Then for all $T > 0$ the following priori estimate is fulfilled:
$$\mu^{\frac{1}{\rho}}\|u\|_{\tilde{L}^{\rho}_{T}(\dot{B}^{s+\frac{2}{\rho}}_{p,r})}\lesssim  \|u_{0}\|_{\dot{B}^{s}_{p,r}}+\mu^{\frac{1}{\rho_1}-1}\|f\|_{\tilde{L}_{T}^{\rho_{1}}(\dot{B}^{s-2+\frac{2}{\rho_{1}}}_{p,r})}.$$
\end{prop}

\subsection{{\bf Hybrid Besov space}}
\hspace*{\fill}

Owing to the low frequencies, the homogeneous Besov spaces fail to have the embedding:
$\dot{B}^{s}_{p,r_{1}}\hookrightarrow\dot{B}^{t}_{p,r_{2}}$ for $s>t$. This motivates the definition of Hybrid Besov spaces, which were introduced by
Danchin \cite{D} and generalized a little bit by Haspot \cite{H2}.

\begin{defn}\label{defn hybrid}
Let $s,t\in \mathbb{R}$, $p,q,r_{1},r_{2}\in [1,\infty]$. We denote $\dot{B}^{s,t}_{(p,r_{1}),(q,r_{2})}$ by the space of functions $f\in\mathcal{S}'_{0}$ equipped with the following
norm:

\begin{eqnarray*}\|f\|_{\dot{B}^{s,t}_{(p,r_{1}),(q,r_{2})}}=\Big\{\sum_{j\geq j_{0}}2^{sjr_{1}}\|\dot{\Delta}_{j}f\|^{r_{1}}_{L^{p}}\Big\}^{\frac{1}{r_{1}}}+\Big\{\sum_{j<
j_{0}}2^{tjr_{2}}\|\dot{\Delta}_{j}f\|^{r_{2}}_{L^{q}}\Big\}^{\frac{1}{r_{2}}},
 \end{eqnarray*}
for some integer $j_{0}$.

\end{defn}
For convenience, we would like to write $\|f\|_{\dot{B}^{s,t}_{(p,r_{1}),(q,r_{2})}}=\|f\|^{h}_{{{\dot B}^{s}}_{p,r_{1}}}+\|f\|^{\ell}_{{{\dot B}^{t}}_{q,r_{2}}}.$ Clearly, $\dot{B}^{s,t}_{(p,r),(q,r)}=\dot{B}^{s,t}_{p,q,r}$, $\dot{B}^{s,s}_{(p,r),(p,r)}=\dot{B}^{s}_{p,r}$. Moreover, one can define the hybrid Chemin-Lerner spaces
$\tilde{L}^{\rho_{1},\rho_{2}}_{T}(\dot{B}^{s,t}_{(p,r_{1}),(q,r_{2})})$ with norm:
 \begin{eqnarray*}\|f\|_{\tilde{L}^{\rho_{1},\rho_{2}}_{T}(\dot{B}^{s,t}_{(p,r_{1}),(q,r_{2})})}=\Big\{2^{sj}\|\dot{\Delta}_{j}f\|_{L^{\rho_{1}}_{T}L^{p}}\Big\}_{l^{r_{1}}_{j\geq
 j_{0}}}+\Big\{2^{tj}\|\dot{\Delta}_{j}f\|_{L^{\rho_{2}}_{T}L^{q}}\Big\}_{l^{r_{2}}_{j< j_{0}}}
 \end{eqnarray*}
for $T>0$. Similarly, we write $\|f\|_{\tilde{L}^{\rho_{1},\rho_{2}}_{T}(\dot{B}^{s,t}_{(p,r_{1}),(q,r_{2})})}=\|f\|^{h}_{\tilde{L}^{\rho_{1}}_{T}(\dot{B}^{s}_{p,r_{1}})}+
\|f\|^{\ell}_{\tilde{L}^{\rho_{2}}_{T}(\dot{B}^{t}_{q,r_{2}})}$. Also,
$\tilde{L}^{\rho,\rho}_{T}(\dot{B}^{s,t}_{(p,r_{1}),(q,r_{2})})=\tilde{L}^{\rho}_{T}(\dot{B}^{s,t}_{(p,r_{1}),(q,r_{2})})$. For notational simplicity, index $T$ is omitted if $T=+\infty$, changing $[0,T]$ to $[0,+\infty)$ in the definition above.

\begin{rem}
If $s\geq t+\frac{d}{p}-\frac{d}{q}$, then $\dot{B}^{s,t}_{(p,r_{1}),(q,r_{2})}=\dot{B}^{s}_{p,r_{1}}\cap\dot{B}^{t}_{q,r_{2}}$. Otherwise, if $s<
    t+\frac{d}{p}-\frac{d}{q}$, then $\dot{B}^{s,t}_{(p,r_{1}),(q,r_{2})}=\dot{B}^{s}_{p,r_{1}}+\dot{B}^{t}_{q,r_{2}}$.
\end{rem}

\subsection{Besov-Gevrey estimates for product and composition}
As the preliminary part, we would also like to list nonlinear estimates involving Besov-Gevrey regularity, which have been shown by \cite{CDX}. First of all, let us give the following elementary conclusions. The reader is referred to \cite{BBT} for those proofs.
\begin{lem}\label{lem5.1} Consider the operator $F:=e^{-(\sqrt{t-s}+\sqrt s-\sqrt t)\Lambda_1}$ for $0\leq s\leq t$. Then $F$ is either the identity operator or is an $L^1$ kernel whose
$L^1$ norm is bounded independent of $s,t$.
\end{lem}

\begin{lem}\label{lem5.2}The operator $F=e^{\frac{1}{2}a\Delta+{\sqrt a}\Lambda_1}$ is a Fourier multiplier which maps boundedly $L^p\to L^p$, $1< p<\infty$, and
its operator norm is uniformly bounded with respect to $a\geq0$.
\end{lem}

Proving the Gevrey regularity of solutions will be based on continuity results for the family
$(\mathcal{B}_{t})_{t\geq 0}$ of bilinear operators defined by
\begin{eqnarray*}
\mathcal{B}_{t}(f,g)(t,x)&=&e^{\sqrt {c_{0}t}\Lambda_1}(e^{-\sqrt {c_{0}t}\Lambda_1}fe^{-\sqrt {c_{0}t}\Lambda_1}g)(x)
\nonumber\\&=& \frac{1}{(2\pi)^{2d}}\int_{\mathbb{R}^d}\int_{\mathbb{R}^d}e^{ix\cdot(\xi+\eta)}e^{\sqrt {c_{0}t}(|\xi+\eta|_{1}-|\xi|_{1}-|\eta|_{1})}\widehat f(\xi)\widehat g(\eta)d\xi d\eta
\end{eqnarray*}
for some $c_0>0$. Following from \cite{BBT} and \cite{L-book}, one can introduce the following operators acting on functions depending
on one real variable:
$$K_{1}f=\frac{1}{2\pi}\int_{0}^{\infty}e^{ix\xi}\widehat f(\xi)d\xi,$$
$$K_{-1}f=\frac{1}{2\pi}\int_{-\infty}^{0}e^{ix\xi}\widehat f(\xi)d\xi,$$
and define $L_{a,1}$ and $L_{a,-1}$ as follows:
$$L_{a,1}f=f\quad \mbox{and}\quad L_{a,-1}f=\frac{1}{2\pi}\int_{R^d}e^{ix\xi}e^{-2a|\xi|}\widehat f(\xi)d\xi.$$
Set
$$Z_{t,\alpha,\beta}=K_{\beta_{1}}L_{\sqrt{c_{0}t},\alpha_{1}\beta_{1}}\otimes...\otimes K_{\beta_{d}}L_{\sqrt{c_{0}t},\alpha_{d}\beta_{d}}
\quad \mbox{and}\quad K_{\alpha}=K_{\alpha_{1}}\otimes ...\otimes K_{\alpha_{d}}$$
for $t\geq 0$, $\alpha=(\alpha_{1}, ... ,\alpha_{d})$ and $\beta=(\beta_{1}, ...,\beta_{d})\in \lbrace-1,1\rbrace^d$.
Then it follows that
$$\mathcal{B}_{t}(f,g)=\sum_{(\alpha,\beta,\gamma)\in({\{-1,1\}^{d})}^3}K_{\alpha}(Z_{t,\alpha,\beta}fZ_{t,\alpha,\gamma}g).$$
It is not difficult to see that $K_{\alpha},Z_{t,\alpha, \beta}$ are linear combinations of smooth homogeneous of
degree zero Fourier multipliers, which are bounded on $L^p$ for $1<p<\infty$.
\begin{lem}\label{lem5.3} (\cite{BBT,L-book})
For any $1<p,p_{1},p_{2}<\infty$ with $\frac{1}{p}=\frac{1}{p_{1}}+\frac{1}{p_{2}}$, we have for some constant $C$
independent of $t\geq0$,
$$\|\mathcal{B}_{t}(f,g)\|_{L^p}\leq C\|f\|_{L^{p_{1}}}\|g\|_{L^{p_{2}}}.$$
\end{lem}
Based on Lemma \ref{lem5.3}, by using Bony's decomposition, one can prove the following product estimates with Gevrey multiplier. The interested reader is referred to \cite{CDX} for the proof details.

\begin{prop}\label{prop5.1}
Let $1< p <\infty$. If $s_{1},s_{2}\leq\frac{d}{p}$, $s_{1}+s_{2}>d\max\{0,\frac{2}{p}-1\}$, then it holds that
$$\|e^{\sqrt {c_{0}t}\Lambda_1}(fg)\|_{\dot{B}^{s_{1}+s_{2}-\frac{d}{p}}_{p,1}}\lesssim\|F\|_{\dot{B}^{s_{1}}_{p,1}}\|G\|_{\dot{B}^{s_{2}}_{p,1}}.$$
Moreover, if $s_{1}\leq\frac{d}{p}$, $s_{2}<\frac{d}{p}$, $s_{1}+s_{2}\geq d\max\{0,\frac{2}{p}-1\}$, then it holds that
$$\|e^{\sqrt {c_{0}t}\Lambda_1}(fg)\|_{\dot{B}^{s_{1}+s_{2}-\frac{d}{p}}_{p,\infty}}\lesssim\|F\|_{\dot{B}^{s_{1}}_{p,1}}\|G\|_{\dot{B}^{s_{2}}_{p,\infty}},$$
where $F=e^{\sqrt {c_{0}t}\Lambda_1}{f}$, $G=e^{\sqrt {c_{0}t}\Lambda_1}{g}$.
\end{prop}
Finally, we state the Gevrey estimate for composition to end this section.
\begin{prop}\label{prop5.2}
Let $F$ be a real analytic function in a neighborhood of 0, such that
$F(0) = 0$. Let $1< p<\infty$, $1\leq r\leq\infty$ and $-d\mathrm{min}(\frac{1}{p},\frac{1}{p'})<s<\frac{d}{p}$ with $\frac{1}{p'}=1-\frac{1}{p}$, then there exist two constants $R_{0}$ and $D$ depending only on $d,p$ and $F$ such that if for some $T>0$,
$$\|e^{\sqrt {c_{0}t}\Lambda_1}z\|_{\tilde{L}^{\infty}(\dot{B}^{\frac{d}{p}}_{p,1})}\lesssim R_{0}$$
then
$$\|e^{\sqrt {c_{0}t}\Lambda_1}F(z)\|_{\tilde{L}^{\theta}_{T}(\dot{B}^{s}_{p,r})}\leq D\|e^{\sqrt {c_{0}t}\Lambda_1}z\|_{\tilde{L}^{\theta}_{T}(\dot{B}^{s}_{p,r})}$$
for $1\leq \theta\leq\infty$. Moreover, the case $s=\frac{d}{p}$ holds true for $r=1$.
\end{prop}

Let us end this section with a variant of the previous result which was firstly proved in \cite{CDX}:
\begin{prop}\label{prop5.3}
Let $F$ be a real analytic function in a neighborhood of 0, such that
$F(0) = 0$. Let $1< p<\infty$, $1\leq r\leq\infty$ and $-dmin(\frac{1}{p},\frac{1}{p'})<s\leq\frac{d}{p}$ with $\frac{1}{p'}=1-\frac{1}{p}$, there exist two constants $R_{0}$ and $\bar{D}$ depending only on $d,p$ and $F$ such that if for some $T>0$,
$$\max_{i=1,2}\|e^{\sqrt {c_{0}t}\Lambda_1}z_{i}\|_{\tilde{L}_{T}^{\infty}(\dot{B}^{\frac{d}{p}}_{p,1})}\lesssim R_{0}$$
then
$$\|e^{\sqrt {c_{0}t}\Lambda_1}\big(F(z_{1})-F(z_{2})\big)\|_{\tilde{L}^{\theta}_{T}(\dot{B}^{s}_{p,r})}\leq \bar{D}\|e^{\sqrt {c_{0}t}\Lambda_1}(z_{1}-z_{2})\|_{\tilde{L}^{\theta}_{T}(\dot{B}^{s}_{p,r})}$$
for $1\leq \theta\leq\infty$.  Moreover, the case $s=\frac{d}{p}$ holds true for $r=1$.
\end{prop}

\section{The global well-posedness and Gevrey analyticity}\setcounter{equation}{0}
In this section, our central task is to show the proof of Theorem \ref{thm2.1}. It is observed that the linearized system of \eqref{linearized} is \textit{purely} parabolic in both high frequencies and low frequencies when $\bar{\nu}^2\geq4\bar{\kappa}$. We shall develop the $L^p$ energy argument
in terms of the effective velocity \eqref{effective}, which leads to the global well-posedness and Gevrey analyticity of (\ref{linearized})-(\ref{composite1}).

\subsection{Global well-posedness}
First of all, let us give a priori estimate for the solutions to (\ref{linearized})-(\ref{composite1}).
\begin{prop}\label{propp}
Suppose that $\bar{\nu}^2\geq4\bar{\kappa}$. Let $1 \leq q\leq p\leq 2q$, $p<d$ and $$\frac{1}{q}<\frac{1}{p}+\frac{2}{d}.$$
Suppose that $(a,m)$ is the solution of system (\ref{linearized})-(\ref{composite1}). There exists a constant $R>0$ that if
$$\|a\|_{\tilde{L}^{\infty}_{T}(\dot{B}^{\frac{d}{p}}_{p,1})}\leq R,$$
then
\begin{equation}\label{R-E3.1}
\|(a,m)\|_{E^{p,q}_{T}}\lesssim \|(\nabla a_{0},
m_{0})\|_{\dot{B}^{\frac{d}{p}-1,\frac{d}{q}-3}_{(p,1),(q,\infty)}}+C_{R}\big(1+\|(a,m)\|_{E^{p,q}_{T}}\big)\|(a,m)\|^{2}_{E^{p,q}_{T}}\end{equation}
 holds for any $T>0$, where the constant $C_{R}>0$ is dependent of $R$.
\end{prop}
\begin{proof}
\noindent \textit{Step 1: The high frequency analysis}

By using the Leray Projector $\mathcal{P}\triangleq \mathrm{Id}-\nabla(-\Delta)^{-1}\mathrm{div}$, one can write $m=\mathcal{P}m+\cQ m$, which is the sum of
the incompressible part $\mathcal{P}m$ and the compressible part $\cQ m$. From
\eqref{spectrum}, we see that $\mathcal{P}m$ fulfills a mere heat equation:
\begin{equation}\label{incompressible}\partial_{t}\mathcal{P}m-\bar{\mu}\Delta \mathcal{P}m =\mathcal{P}g.\end{equation}
Hence, it follows from Proposition \ref{prop2.2} that
\begin{equation}\label{4.1}
\|\mathcal{P}m\|^{h}_{\tilde{L}_{T}^{\infty}(\dot{B}^{\frac{d}{p}-1}_{p,1})}+\|\mathcal{P}m\|^{h}_{\tilde{L}_{T}^{1}(\dot{B}^{\frac{d}{p}+1}_{p,1})}\lesssim
\|\mathcal{P}m_{0}\|^{h}_{\dot{B}^{\frac{d}{p}-1}_{p,1}}+
\|\mathcal{P}g\|^{h}_{\tilde{L}_{T}^{1}(\dot{B}^{\frac{d}{p}-1}_{p,1})}.
\end{equation}
In order to handle the coupling between $a$ and $\cQ m,$ inspired by \cite{H2}, we introduce the effective velocity
$$w\triangleq \cQ m+\alpha \nabla a$$
where $\alpha$ is to be confirmed. Note that $\nabla\div m=\Delta\cQ m$, it follows from (1.4) that
$$\left\{\begin{array}{l}\d_t\nabla a+ \Delta\cQ m=0,\\[1ex]
\d_t\cQ m-\bar{\nu}\Delta\cQ m-\bar{\kappa}\Delta\nabla  a=\cQ g.
\end{array}\right.$$
Consequently, we arrive at
\begin{equation}
\d_tw-(\bar{\nu}-\alpha)\Delta w=\cQ g,\label{4.2}
\end{equation}
where $\alpha$ satisfies $\alpha=\frac{\bar{\kappa}}{\bar{\nu}-\alpha}$. Solving the equality implies that $\alpha=\frac{1}{2}(\bar{\nu}\pm\sqrt{\bar{\nu}^2-4\bar{\kappa}}),$
furthermore, that $\bar{\nu}-\alpha=\frac{1}{2}(\bar{\nu}\mp\sqrt{\bar{\nu}^2-4\bar{\kappa}})>0.$ Hence, employing Proposition \ref{prop2.2}  to get
\begin{equation}
\|w\|^{h}_{\tilde{L}^{\infty}_{T}(\dot{B}^{\frac{d}{p}-1}_{p,1})}+
\|w\|^{h}_{\tilde{L}^{1}_{T}(\dot{B}^{\frac{d}{p}+1}_{p,1})}\lesssim  \|w_0\|^{h}_{\dot{B}^{\frac{d}{p}-1}_{p,1}}+\|\cQ
g\|^{h}_{\tilde{L}^{1}_{T}(\dot{B}^{\frac{d}{p}-1}_{p,1})}.\label{4.3}
\end{equation}

On the other hand, bounding estimate $\cQ m$, we use the fact
\begin{equation} \nabla a=\frac{w-\cQ m}{\alpha},\label{4.4}\end{equation}
so that the equation for $\cQ m$ can be rewritten as
\begin{equation}
\d_t\cQ m-\alpha \Delta \cQ m=\frac{\bar{\kappa}}{\alpha}\Delta w+\cQ g.\label{4.5}
\end{equation}
Hence, one can again take advantage of (\ref{4.3}) and Proposition \ref{prop2.2} to arrive at
\begin{multline}
\|\cQ m\|^{h}_{\tilde{L}^{\infty}_{T}(\dot{B}^{\frac{d}{p}-1}_{p,1})}+
\|\cQ m\|^{h}_{\tilde{L}^{1}_{T}(\dot{B}^{\frac{d}{p}+1}_{p,1})}\lesssim  \|\cQ m_0\|^{h}_{\dot{B}^{\frac{d}{p}-1}_{p,1}}+\Big\|\frac{\kappa}{\alpha}\Delta w+\cQ
g\Big\|^{h}_{\tilde{L}^{1}_{T}(\dot{B}^{\frac{d}{p}-1}_{p,1})}\\ \lesssim\|\cQ m_0\|^{h}_{\dot{B}^{\frac{d}{p}-1}_{p,1}}+\|w_0\|^{h}_{\dot{B}^{\frac{d}{p}-1}_{p,1}}+\|\cQ
g\|^{h}_{\tilde{L}^{1}_{T}(\dot{B}^{\frac{d}{p}-1}_{p,1})}.\label{4.6}
\end{multline}
Notice that
\begin{multline}
\|\nabla a\|^{h}_{\tilde{L}^{\infty}_{T}(\dot{B}^{\frac{d}{p}-1}_{p,1})\cap\tilde{L}^{1}_{T}(\dot{B}^{\frac{d}{p}+1}_{p,1})}\lesssim  \|\frac{w-\cQ m}{\alpha}\|^{h}_{\tilde{L}^{\infty}_{T}(\dot{B}^{\frac{d}{p}-1}_{p,1})\cap\tilde{L}^{1}_{T}(\dot{B}^{\frac{d}{p}+1}_{p,1})}\\ \lesssim\|w\|^{h}_{\tilde{L}^{\infty}_{T}(\dot{B}^{\frac{d}{p}-1}_{p,1})\cap\tilde{L}^{1}_{T}(\dot{B}^{\frac{d}{p}+1}_{p,1})}+\|m\|^{h}_{\tilde{L}^{\infty}_{T}(\dot{B}^{\frac{d}{p}-1}_{p,1})\cap\tilde{L}^{1}_{T}(\dot{B}^{\frac{d}{p}+1}_{p,1})}.\label{4.6}
\end{multline}
Therefore it follows from (\ref{4.3}) and (\ref{4.6}) that
\begin{multline}\label{4.7}
\|(\nabla a, \cQ m)\|^{h}_{\tilde{L}^{\infty}_{T}(\dot{B}^{\frac{d}{p}-1}_{p,1})}+\|(\nabla a, \cQ m)\|^{h}_{\tilde{L}^{1}_{T}(\dot{B}^{\frac{d}{p}+1}_{p,1})}
\\ \lesssim \|(\nabla a_0, \cQ m_0)\|^{h}_{\dot{B}^{\frac{d}{p}-1}_{p,1}}+\|\cQ g\|^{h}_{\tilde{L}^{1}_{T}(\dot{B}^{\frac{d}{p}-1}_{p,1})}.
\end{multline}

\noindent \textit{Step 2: The low frequency analysis}

To control the evolution of norm $\|a\|_{L^{1}_{T}(\dot{B}^{d/p}_{p,1})}$ at the low frequencies, we need to establish the $\tilde{L}_{T}^{\infty}(\dot{B}^{d/q-3}_{q,\infty})(1\leq q\leq p)$ estimates of $(\nabla a,m)$. The effective velocity
still plays a key role in the low-frequency estimates. Again for the incompressible part $\mathcal{P}m$, one can restrict Proposition
\ref{prop2.2} to the low-frequency part and get
\begin{equation}\label{4.8}
\|\mathcal{P}m\|^{\ell}_{\tilde{L}_{T}^{\infty}(\dot{B}^{\frac{d}{q}-3}_{q,\infty})}+\|\mathcal{P}m\|^{\ell}_{\tilde{L}_{T}^{1}(\dot{B}^{\frac{d}{q}-1}_{q,\infty})}\lesssim
\|\mathcal{P}m_{0}\|^{\ell}_{\dot{B}^{\frac{d}{q}-3}_{q,\infty}}+
\|\mathcal{P}g\|^{\ell}_{\tilde{L}_{T}^{1}(\dot{B}^{\frac{d}{q}-3}_{q,\infty})}
\end{equation}
for $1\leq q\leq p$.

Regarding $a$ and $\mathcal{Q}m$,  by using the effective velocity $w=\cQ m+\alpha \nabla a$ with $\alpha=\frac{\bar{\kappa}}{\bar{\nu}-\alpha}$ as in high frequencies, one can obtain
\begin{equation}\label{4.10}
\|w\|^{\ell}_{\tilde{L}^{\infty}_{T}(\dot{B}^{\frac{d}{q}-3}_{q,\infty})}+
\|w\|^{\ell}_{\tilde{L}^{1}_{T}(\dot{B}^{\frac{d}{q}-1}_{q,\infty})}\lesssim  \|w_0\|^{\ell}_{\dot{B}^{\frac{d}{q}-3}_{q,\infty}}+\|\cQ
g\|^{\ell}_{\tilde{L}^{1}_{T}(\dot{B}^{\frac{d}{q}-3}_{q,\infty})}.
\end{equation}
It follows from (\ref{4.5}) that
\begin{eqnarray}\label{4.12}
&&\|\cQ m\|^{\ell}_{\tilde{L}^{\infty}_{T}(\dot{B}^{\frac{d}{q}-3}_{q,\infty})}+
\|\cQ m\|^{\ell}_{\tilde{L}^{1}_{T}(\dot{B}^{\frac{d}{q}-1}_{q,\infty})}\nonumber\\&\lesssim&\|\cQ
m_0\|^{\ell}_{\dot{B}^{\frac{d}{q}-3}_{q,\infty}}+\|w_0\|^{\ell}_{\dot{B}^{\frac{d}{q}-3}_{q,\infty}}+\|\cQ g\|^{\ell}_{\tilde{L}^{1}_{T}(\dot{B}^{\frac{d}{q}-3}_{q,\infty})}.
\end{eqnarray}
Consequently, by combining (\ref{4.10}) and (\ref{4.12}), we arrive at
\begin{eqnarray}\label{4.13}
\|(\nabla a,\mathcal{Q}m)\|^{\ell}_{\tilde{L}_{T}^{\infty}(\dot{B}^{\frac{d}{q}-3}_{q,\infty})\bigcap\tilde{L}_{T}^{1}(\dot{B}^{\frac{d}{q}-1}_{q,\infty})}
\lesssim\|(\nabla a_{0},\mathcal{Q}m_{0})\|^{\ell}_{\dot{B}^{\frac{d}{q}-3}_{q,\infty}}+
\|\mathcal{Q}g\|^{\ell}_{\tilde{L}_{T}^{1}(\dot{B}^{\frac{d}{q}-3}_{q,\infty})}.
\end{eqnarray}

\noindent \textit{Step 3: The nonlinear estimates}

Putting inequalities \eqref{4.1}, \eqref{4.7}-\eqref{4.8} and \eqref{4.13} together, we conclude that
\begin{eqnarray}\label{4.13c}
\|(a,m)\|_{E^{p,q}_{T}} \lesssim \|(\nabla a_{0}, m_{0})\|_{\dot{B}^{\frac{d}{p}-1,\frac{d}{q}-3}_{(p,1),(q,\infty)}} + \|g\|^{h}_{\tilde{L}_{T}^{1}(\dot{B}^{\frac{d}{p}-1}_{p,1})}+\|g\|^{\ell}_{\tilde{L}_{T}^{1}(\dot{B}^{\frac{d}{q}-3}_{q,\infty})}.
\end{eqnarray}

In what follows, let us handle those nonlinear terms in $g$. We start with the high-frequency estimates. For $g_{1}$, by Proposition \ref{prop3.2}, we have for $1\leq p< \infty$:
\begin{eqnarray*}
\|g_{1}\|^{h}_{\tilde{L}^{1}_{T}(\dot{B}^{\frac{d}{p}-1}_{p,1})}\lesssim\|\big(Q(a)-1\big)m\otimes m\|^{h}_{\tilde{L}^{1}_{T}(\dot{B}^{\frac{d}{p}}_{p,1})}
\lesssim(1+\|a\|_{\tilde{L}^{\infty}_{T}(\dot{B}^{\frac{d}{p}}_{p,1})})\|m\|^{2}_{\tilde{L}^{2}_{T}(\dot{B}^{\frac{d}{p}}_{p,1})}.
\end{eqnarray*}
Notice that the embedding $\dot{B}^{s+\frac{d}{q}-\frac{d}{p}}_{q,\infty}\hookrightarrow \dot{B}^{t}_{p,1}(q\leq p$ and $s<t)$ on low frequencies,
we deduct that
\begin{eqnarray*}
\|a\|_{\tilde{L}^{\infty}_{T}(\dot{B}^{\frac{d}{p}}_{p,1})}\lesssim\|a\|^{h}_{\tilde{L}^{\infty}_{T}(\dot{B}^{\frac{d}{p}}_{p,1})}
+\|a\|^{\ell}_{\tilde{L}^{\infty}_{T}(\dot{B}^{\frac{d}{q}-2}_{q,\infty})}\lesssim \|a\|_{E^{p,q}_{T}}.
\end{eqnarray*}
Thanks to the interpolation and embedding,
\begin{eqnarray*}
\|m\|_{\tilde{L}^{2}_{T}(\dot{B}^{\frac{d}{p}}_{p,1})}\lesssim \|m\|_{\tilde{L}^{\infty}_{T}(\dot{B}^{\frac{d}{p}-1}_{p,1})}^{\frac 12}\|m\|_{\tilde{L}^{1}_{T}(\dot{B}^{\frac{d}{p}+1}_{p,1})}^{\frac 12}\lesssim \|m\|_{E^{p,q}_{T}}.
\end{eqnarray*}
Consequently, we arrive at
\begin{eqnarray}\label{g1}
\|g_{1}\|^{h}_{\tilde{L}^{1}_{T}(\dot{B}^{\frac{d}{p}-1}_{p,1})}\lesssim(1+\|(a,m)\|_{E^{p,q}_{T}})\|(a,m)\|^{2}_{E^{p,q}_{T}}.
\end{eqnarray}
Similarly,
\begin{eqnarray}\label{glll}
\|g_{2}\|^{h}_{\tilde{L}^{1}_{T}(\dot{B}^{\frac{d}{p}-1}_{p,1})}&\lesssim& \|a\|_{\tilde{L}^{2}_{T}(\dot{B}^{\frac{d}{p}+1}_{p,1})}\|m\|_{\tilde{L}^{2}_{T}(\dot{B}^{\frac{d}{p}}_{p,1})}+
\|a\|_{\tilde{L}^{\infty}_{T}(\dot{B}^{\frac{d}{p}}_{p,1})}\|m\|_{\tilde{L}^{1}_{T}(\dot{B}^{\frac{d}{p}+1}_{p,1})}\\
\nonumber&\lesssim&\|(a,m)\|^{2}_{E^{p,q}_{T}}
\end{eqnarray}
and
\begin{multline}
\|g_{3}\|^{h}_{\tilde{L}^{1}_{T}(\dot{B}^{\frac{d}{p}-1}_{p,1})}\lesssim\|\tilde{\mu}(a)D((1-Q(a))m\big)\|^{h}_{\tilde{L}^{1}_{T}(\dot{B}^{\frac{d}{p}}_{p,1})}+
\|\tilde{\lambda}(a)\mathrm{div}\big((1-Q(a))m)\big)\|^{h}_{\tilde{L}^{1}_{T}(\dot{B}^{\frac{d}{p}}_{p,1})}\\
\lesssim\|a\|_{\tilde{L}^{\infty}_{T}(\dot{B}^{\frac{d}{p}}_{p,1})}\Big\{(1+\|a\|_{\tilde{L}^{2}_{T}(\dot{B}^{\frac{d}{p}+1}_{p,1})})\|m\|_{\tilde{L}^{2}_{T}(\dot{B}^{\frac{d}{p}}_{p,1})}
+(1+\|a\|_{\tilde{L}^{\infty}_{T}(\dot{B}^{\frac{d}{p}}_{p,1})})\|m\|_{\tilde{L}^{1}_{T}(\dot{B}^{\frac{d}{p}+1}_{p,1})})\Big\}\\
\lesssim(1+\|(a,m)\|_{E^{p,q}_{T}})\|(a,m)\|^{2}_{E^{p,q}_{T}}.
\end{multline}

Owing to the fact that $P'(\rho^{*})=0$, bounding the nonlinear pressure is more elaborate. Indeed,
\begin{equation}
\|g_{4}\|^{h}_{\tilde{L}^{1}_{T}(\dot{B}^{\frac{d}{p}-1}_{p,1})}\lesssim
\|a\|^{2}_{\tilde{L}^{2}_{T}(\dot{B}^{\frac{d}{p}}_{p,1})}.
\end{equation}
It follows from the interpolation and embedding that
\begin{equation*}
\|a\|_{\tilde{L}^{2}_{T}(\dot{B}^{\frac{d}{p}}_{p,1})}\leq \|a^{h}\|_{\tilde{L}^{2}_{T}(\dot{B}^{\frac{d}{p}}_{p,1})}+\|a^{\ell}\|_{\tilde{L}^{2}_{T}(\dot{B}^{\frac{d}{p}}_{p,1})},
\end{equation*}
Clearly, it holds
$$\|a^{h}\|_{\tilde{L}^{2}_{T}(\dot{B}^{\frac{d}{p}}_{p,1})}\lesssim \|a\|^{h}_{\tilde{L}^{\infty}_{T}(\dot{B}^{\frac{d}{p}}_{p,1})}+\|a\|^{h}_{\tilde{L}^{1}_{T}(\dot{B}^{\frac{d}{p}+2}_{p,1})}$$
and
$$\|a^{\ell}\|_{\tilde{L}^{2}_{T}(\dot{B}^{\frac{d}{p}}_{p,1})}\lesssim \|a^{\ell}\|_{\tilde{L}^{2}_{T}(\dot{B}^{\frac{d}{q}-1}_{q,1})}\lesssim
\|a\|^{\ell}_{\tilde{L}^{\infty}_{T}(\dot{B}^{\frac{d}{q}-2}_{q,\infty})}+\|a\|^{\ell}_{\tilde{L}^{1}_{T}(\dot{B}^{\frac{d}{q}}_{q,\infty})}.$$
Consequently, we are led to
\begin{equation*}
\|g_{4}\|^{h}_{\tilde{L}^{1}_{T}(\dot{B}^{\frac{d}{p}-1}_{p,1})}\lesssim\|(a,m)\|^{2}_{E^{p,q}_{T}}.
\end{equation*}
Regarding those capillary terms, it follows that
\begin{eqnarray}\label{gll}
\|g_{5}\|^{h}_{\tilde{L}^{1}_{T}(\dot{B}^{\frac{d}{p}-1}_{p,1})}\lesssim
\|a\|_{\tilde{L}^{\infty}_{T}(\dot{B}^{\frac{d}{p}}_{p,1})}{\|a\|_{\tilde{L}^{1}_{T}(\dot{B}^{\frac{d}{p}+2}_{p,1})}}\lesssim\|(a,m)\|^{2}_{E^{p,q}_{T}}
\end{eqnarray}
and
\begin{multline}\label{opq}
\|g_{6}\|^{h}_{\tilde{L}^{1}_{T}(\dot{B}^{\frac{d}{p}-1}_{p,1})}\lesssim\|\big(\tilde{\kappa}_{2}(a)+\check{\kappa}\big)|\nabla
a|^{2}\|^{h}_{\tilde{L}^{1}_{T}(\dot{B}^{\frac{d}{p}}_{p,1})}+
\|\big(\tilde{\kappa}_{3}(a)+\bar{\kappa}\big)\nabla a\otimes\nabla a\|^{h}_{\tilde{L}^{1}_{T}(\dot{B}^{\frac{d}{p}}_{p,1})}\\
\lesssim(1+\|a\|_{\tilde{L}^{\infty}_{T}(\dot{B}^{\frac{d}{p}}_{p,1})})\|a\|^{2}_{\tilde{L}^{2}_{T}(\dot{B}^{\frac{d}{p}+1}_{p,1})}
\lesssim(1+\|(a,m)\|_{E^{p,q}_{T}})\|(a,m)\|^{2}_{E^{p,q}_{T}}.
\end{multline}
Therefore, combing (\ref{g1})-(\ref{opq}), we deduce that
\begin{eqnarray}\label{h-f}
\|g\|^{h}_{\tilde{L}^{1}_{T}(\dot{B}^{\frac{d}{p}-1}_{p,1})}\lesssim\big(1+\|(a,m)\|_{E^{p,q}_{T}}\big)\|(a,m)\|^{2}_{E^{p,q}_{T}}.
\end{eqnarray}

Now what left is to deal with the norm $\|g\|^{\ell}_{\tilde{L}_{T}^{1}(\dot{B}^{\frac{d}{q}-3}_{q,\infty})}$. For that end, some product estimates need to be developed. Precisely,
\begin{prop}\label{cor1.01}
Let $1 \leq q\leq p\leq 2q$, $p<d$ and $$\frac{1}{q}<\frac{1}{p}+\frac{2}{d}.$$
Then it holds that
$$\|ab\|_{\dot{B}^{\frac{d}{q}-2}_{q,\infty}}\lesssim\|a\|_{\dot{B}^{\frac{d}{p}-2}_{p,\infty}}\|b\|_{\dot{B}^{\frac{d}{p}}_{p,\infty}}+
\|b\|_{\dot{B}^{\frac{d}{p}-2}_{p,\infty}}\|a\|_{\dot{B}^{\frac{d}{p}}_{p,\infty}}.$$
\end{prop}

Proposition \ref{cor1.01} can be directly deduced by Lemma \ref{lem key} if we take $s=\frac{d}{p}$, $k_{1}=2$. Precisely,
the condition $d\geq3$ and $p< d$ are equivalent to $$\frac dp> 2-d\min\Big(\frac 1p,\frac 1{p'}\Big).$$
Also, $q\leq p$ and $\frac{1}{q}<\frac{1}{p}+\frac{2}{d} \Leftrightarrow d\max(0,1/q-1/p)<2.$  A version in Chemin-Lerner spaces can be immediately reached,  whereas the time exponent fulfills H\"{o}lder inequality.
\begin{cor}\label{cor1.1}
Let $\rho,\rho_{1},\rho_{2},\bar{\rho}_{1},\bar{\rho}_{2}\in[1,\infty]$ with $\frac{1}{\rho}=\frac{1}{\rho_{1}}+\frac{1}{\rho_{2}}=\frac{1}{\bar{\rho}_{1}}+\frac{1}{\bar{\rho}_{2}}$. Let $1 \leq q\leq p\leq 2q$, $p<d$ and $$\frac{1}{q}<\frac{1}{p}+\frac{2}{d}.$$
Then it holds that
$$\|ab\|_{\tilde{L}^{\rho}_{T}(\dot{B}^{\frac{d}{q}-2}_{q,\infty})}\lesssim\|a\|_{\tilde{L}^{\rho_{1}}_{T}(\dot{B}^{\frac{d}{p}-2}_{p,\infty})}
\|b\|_{\tilde{L}^{\rho_{2}}_{T}(\dot{B}^{\frac{d}{p}}_{p,\infty})}+
\|b\|_{\tilde{L}^{\bar{\rho}_{1}}_{T}(\dot{B}^{\frac{d}{p}-2}_{p,\infty})}
\|a\|_{\tilde{L}^{\bar{\rho}_{2}}_{T}(\dot{B}^{\frac{d}{p}}_{p,\infty})}$$
for any $T>0$.
\end{cor}

Let us start to estimate $\|g\|^{\ell}_{\tilde{L}_{T}^{1}(\dot{B}^{\frac{d}{q}-3}_{q,\infty})}$. First of all, we focus on the pressure term $g_{4}$,
which is a lower-order term with respect to the scaling of Korteweg system. In order to get the control of density in low frequencies, the stronger
regularity assumption ($s=d/q-3$) is posted. It follows from Corollary \ref{cor1.1} that
\begin{multline*}
\|\nabla\big(aG(a)\big)\|^{\ell}_{\tilde{L}_{T}^{1}(\dot{B}^{\frac{d}{q}-3}_{q,\infty})}
\lesssim\|aG(a)\|^{\ell}_{\tilde{L}_{T}^{1}(\dot{B}^{\frac{d}{q}-2}_{q,\infty})}
\lesssim\|a\|_{\tilde{L}_{T}^{\infty}(\dot{B}^{\frac{d}{p}-2}_{p,\infty})}\|G(a)\|_{\tilde{L}_{T}^{1}(\dot{B}^{\frac{d}{p}}_{p,\infty})}\\
+\|G(a)\|_{\tilde{L}_{T}^{\infty}(\dot{B}^{\frac{d}{p}-2}_{p,\infty})}\|a\|_{\tilde{L}_{T}^{1}(\dot{B}^{\frac{d}{p}}_{p,\infty})}.
\end{multline*}
By Proposition \ref{prop2.25}, it holds that for $d\geq3$ and $1\leq p<d$,
\begin{eqnarray*}
\|G(a)\|_{\tilde{L}_{T}^{\infty}(\dot{B}^{\frac{d}{p}-2}_{p,\infty})}
\leq C_{R}\|a\|_{\tilde{L}_{T}^{\infty}(\dot{B}^{\frac{d}{p}-2}_{p,\infty})},
\end{eqnarray*}
where $C_{R}>0$ is some constant depending on $R$. To bound $\|G(a)\|_{\tilde{L}_{T}^{1}(\dot{B}^{\frac{d}{p}}_{p,\infty})}$, we write
$$G(a)=G'(0)a+\tilde{G}(a)a.$$
where $\tilde{G}(a)$ is a smooth function. Hence, by Proposition \ref{prop2.25}, there holds
$$\|G(a)\|_{\tilde{L}_{T}^{1}(\dot{B}^{\frac{d}{p}}_{p,\infty})}\lesssim\|a\|_{\tilde{L}_{T}^{1}(\dot{B}^{\frac{d}{p}}_{p,\infty})}+
\|a\|^{2}_{\tilde{L}_{T}^{2}(\dot{B}^{\frac{d}{p}}_{p,1})}.$$
Clearly, by embedding relationship, there holds
\begin{eqnarray}
\|a\|_{\tilde{L}_{T}^{2}(\dot{B}^{\frac{d}{p}}_{p,1})}&\lesssim&\|a^{\ell}\|_{\tilde{L}_{T}^{2}(\dot{B}^{\frac{d}{q}-1}_{q,\infty})}+\|a^{h}\|_{\tilde{L}_{T}^{2}(\dot{B}^{\frac{d}{p}+1}_{p,1})}.
\end{eqnarray}
Hence, by the following interpolation in the Chemin-Lerner space:
$$\|f\|_{\tilde{L}_{T}^{\rho}(\dot{B}_{p,r}^{\theta \sigma_{1}+(1-\theta )\sigma_{2}})}\lesssim \|f\| _{\tilde{L}_{T}^{\rho_1}(\dot{B}_{p,r}^{\sigma_{1}})}^{\theta} \|f\|_{\tilde{L}_{T}^{\rho_2}(\dot{B}_{p,r}^{\sigma_{2}})}^{1-\theta }$$
where $p,r\in[1,\infty]$, $\theta \in [0,1]$ and $\frac{1}{\rho}=\frac{\theta}{\rho_1}+\frac{1-\theta}{\rho_2}$, we have by Cauchy inequality that
$$\|a^{\ell}\|_{\tilde{L}_{T}^{2}(\dot{B}^{\frac{d}{q}-1}_{q,\infty})}\lesssim\|a^{\ell}\|^{\frac{1}{2}}_{\tilde{L}_{T}^{\infty}(\dot{B}^{\frac{d}{q}-2}_{q,\infty})}
\|a^{\ell}\|^{\frac{1}{2}}_{\tilde{L}_{T}^{1}(\dot{B}^{\frac{d}{q}}_{q,\infty})}\lesssim\|(a,m)\|_{E^{p,q}_{T}}.$$
Then similar calculations on $\|a^{h}\|_{\tilde{L}_{T}^{2}(\dot{B}^{\frac{d}{p}+1}_{p,1})}$ finally let we arrive at
\begin{eqnarray}\label{E-3.23}
\|g_{4}\|^{\ell}_{\tilde{L}^{1}(\dot{B}^{\frac{d}{q}-3}_{q,\infty})}\lesssim(1+\|(a,m)\|_{E^{p,q}_{T}})\|(a,m)\|^{2}_{E^{p,q}_{T}}.
\end{eqnarray}

As a matter of fact, other nonlinear terms can be estimated at a similar way. For $g_{1}$, one can employ
Corollary \ref{cor1.1} and Proposition \ref{prop3.2} to get
\begin{eqnarray}\label{E-3.24}
&&\|g_{1}\|^{\ell}_{\tilde{L}_{T}^{1}(\dot{B}^{\frac{d}{q}-3}_{q,\infty})}\\
\nonumber&\lesssim&\|\big(Q(a)-1)m\|_{\tilde{L}_{T}^{\infty}(\dot{B}^{\frac{d}{p}-2}_{p,\infty})}\|m\|_{\tilde{L}_{T}^{1}(\dot{B}^{\frac{d}{p}}_{p,\infty})}+
\|m\|_{\tilde{L}_{T}^{\infty}(\dot{B}^{\frac{d}{p}-2}_{p,\infty})}\|\big(Q(a)-1\big)m\|_{\tilde{L}_{T}^{1}(\dot{B}^{\frac{d}{p}}_{p,\infty})}\\
\nonumber&\lesssim&\big(1+\|a\|_{\tilde{L}_{T}^{\infty}(\dot{B}^{\frac{d}{p}}_{p,1})}\big)\|m\|_{\tilde{L}_{T}^{\infty}(\dot{B}^{\frac{d}{p}-2}_{p,\infty})}
\|m\|_{\tilde{L}_{T}^{1}(\dot{B}^{\frac{d}{p}}_{p,1})}\\
\nonumber&\lesssim&(1+\|(a,m)\|_{E^{p,q}_{T}})\|(a,m)\|^{2}_{E^{p,q}_{T}}.
\end{eqnarray}
For the term with $g_{2}=\bar{\mu}\Delta(Q(a)m)+(\bar{\mu}+\bar{\lambda})\nabla\mathrm{div}(Q(a)m)$, we have
\begin{eqnarray}\label{E-3.25}
&&\|g_{2}\|^{\ell}_{\tilde{L}_{T}^{1}(\dot{B}^{\frac{d}{q}-3}_{q,\infty})}\\
\nonumber&\lesssim&\|\nabla(Q(a)m)\|_{\tilde{L}_{T}^{1}(\dot{B}^{\frac{d}{q}-2}_{q,\infty})}
+\|\mathrm{div}(Q(a)m)\|_{\tilde{L}_{T}^{1}(\dot{B}^{\frac{d}{q}-2}_{q,\infty})}\\
\nonumber&\lesssim&\|Q(a)\|_{\tilde{L}_{T}^{\infty}(\dot{B}^{\frac{d}{p}-2}_{p,\infty})}\|\nabla m\|_{\tilde{L}_{T}^{1}(\dot{B}^{\frac{d}{p}}_{p,\infty})}
+\|\nabla m\|_{\tilde{L}_{T}^{\infty}(\dot{B}^{\frac{d}{p}-2}_{p,\infty})}
\|Q(a)\|_{\tilde{L}_{T}^{1}(\dot{B}^{\frac{d}{p}}_{p,\infty})}\\
\nonumber&&+\|\nabla Q(a)\|_{\tilde{L}_{T}^{\infty}(\dot{B}^{\frac{d}{p}-2}_{p,\infty})}\|m\|_{\tilde{L}_{T}^{1}(\dot{B}^{\frac{d}{p}}_{p,\infty})}
+\|m\|_{\tilde{L}_{T}^{\infty}(\dot{B}^{\frac{d}{p}-2}_{p,\infty})}
\|\nabla Q(a)\|_{\tilde{L}_{T}^{1}(\dot{B}^{\frac{d}{p}}_{p,\infty})}\\
\nonumber&\lesssim&\|(a,m)\|^{2}_{E^{p,q}_{T}}+\|(a,m)\|^{3}_{E^{p,q}_{T}}.
\end{eqnarray}
Regarding $g_{3}$, we estimate $\mathrm{div}(\tilde{\mu}(a)D(Q(a)m))$ as an example. By Corollary \ref{cor1.1}, we deduct that
\begin{eqnarray}\label{E-3.26}
&&\|\mathrm{div}(\tilde{\mu}(a)D(Q(a)m))\|^{\ell}_{\tilde{L}^{1}_{T}(\dot{B}^{\frac{d}{q}-3}_{q,\infty})}\\
\nonumber&\lesssim&\|\tilde{\mu}(a)\|_{\tilde{L}^{\infty}_{T}(\dot{B}^{\frac{d}{p}-2}_{p,\infty})}\|D(Q(a)m)\|_{\tilde{L}^{1}_{T}(\dot{B}^{\frac{d}{p}}_{p,\infty})}+
\|D(Q(a)m)\|_{\tilde{L}^{\infty}_{T}(\dot{B}^{\frac{d}{p}-2}_{p,\infty})}\|\tilde{\mu}(a)\|_{\tilde{L}^{1}_{T}(\dot{B}^{\frac{d}{p}}_{p,\infty})}\\
\nonumber&\lesssim&\|a\|_{E^{p,q}_{T}}
\Big(\|\nabla a\|_{\tilde{L}^{2}_{T}(\dot{B}^{\frac{d}{p}}_{p,1})}\|m\|_{\tilde{L}^{2}_{T}(\dot{B}^{\frac{d}{p}}_{p,1})}
+\|a\|_{\tilde{L}^{\infty}_{T}(\dot{B}^{\frac{d}{p}}_{p,1})}\|\nabla m\|_{\tilde{L}^{2}_{T}(\dot{B}^{\frac{d}{p}}_{p,1})}\\
\nonumber&+&\|\nabla a\|_{\tilde{L}^{2}_{T}(\dot{B}^{\frac{d}{p}-1}_{p,1})}\|m\|_{\tilde{L}^{2}_{T}(\dot{B}^{\frac{d}{p}-1}_{p,1})}
+\|a\|_{\tilde{L}^{\infty}_{T}(\dot{B}^{\frac{d}{p}}_{p,1})}\|\nabla m\|_{\tilde{L}^{\infty}_{T}(\dot{B}^{\frac{d}{p}-2}_{p,1})}\Big)\\
\nonumber&\lesssim&\|(a,m)\|^{3}_{E^{p,q}_{T}}.
\end{eqnarray}
Furthermore, performing similar calculations for $\mathrm{div}(\tilde{\mu}(a)Dm)$ and $\nabla\Big(\tilde{\lambda}(a)\mathrm{div}\big((1-Q(a)\big)m\big)\Big)$ enables us to get
\begin{eqnarray*}
\|g_{3}\|^{\ell}_{\tilde{L}^{1}(\dot{B}^{\frac{d}{q}-3}_{q,\infty})}\lesssim(1+\|(a,m)\|_{E^{p,q}_{T}})\|(a,m)\|^{2}_{E^{p,q}_{T}}.
\end{eqnarray*}
What left is to handle the Korteweg terms in low frequencies. It follows from Corollary \ref{cor1.1} that
\begin{multline}\label{E-3.27}
\|g_{5}\|^{\ell}_{\tilde{L}^{1}_{T}(\dot{B}^{\frac{d}{q}-3}_{q,\infty})}\!\!\!\lesssim\!\|\tilde{\kappa}_{1}(a)\|_{\tilde{L}^{\infty}_{T}(\dot{B}^{\frac{d}{p}-2}_{p,\infty})}\|\Delta a\|_{\tilde{L}^{1}_{T}(\dot{B}^{\frac{d}{p}}_{p,\infty})}\!+\!
\|\Delta a\|_{\tilde{L}^{\infty}_{T}(\dot{B}^{\frac{d}{p}-2}_{p,\infty})}\!\|\tilde{\kappa}_{1}(a)\|_{\tilde{L}^{1}_{T}(\dot{B}^{\frac{d}{p}}_{p,\infty})}\\
\lesssim\|(a,m)\|^{2}_{E^{p,q}_{T}}+\|(a,m)\|^{3}_{E^{p,q}_{T}}.
\end{multline}
For $g_{6}$, it suffices to take a look at the term $\frac{\rho^{*}}{2}\nabla\big(\tilde{\kappa}_{2}(a)|\nabla a|^{2}\big)$. Precisely,
\begin{eqnarray*}
&&\|\frac{\rho^{*}}{2}\nabla\big(\tilde{\kappa}_{2}(a)|\nabla a|^{2}\big)\|^{\ell}_{\tilde{L}^{1}_{T}(\dot{B}^{\frac{d}{q}-3}_{q,\infty})}\\
\nonumber&\lesssim&\|\tilde{\kappa}_{2}(a)\nabla a\|_{\tilde{L}^{\infty}_{T}(\dot{B}^{\frac{d}{p}-2}_{p,\infty})}\|\nabla a\|_{\tilde{L}^{1}_{T}(\dot{B}^{\frac{d}{p}}_{p,\infty})}+
\|\nabla a\|_{\tilde{L}^{\infty}_{T}(\dot{B}^{\frac{d}{p}-2}_{p,\infty})}\|\tilde{\kappa}_{2}(a)\nabla a\|_{\tilde{L}^{1}_{T}(\dot{B}^{\frac{d}{p}}_{p,\infty})}\\
\nonumber&\lesssim&\|a\|_{\tilde{L}^{\infty}_{T}(\dot{B}^{\frac{d}{p}}_{p,1})}\|\nabla a\|_{\tilde{L}^{\infty}_{T}(\dot{B}^{\frac{d}{p}-2}_{p,\infty})}\|\nabla
a\|_{\tilde{L}^{1}_{T}(\dot{B}^{\frac{d}{p}}_{p,1})}\\
\nonumber&\lesssim&\|(a,m)\|^{3}_{E^{p,q}_{T}}.
\end{eqnarray*}
Repeating above calculations to $\frac{\rho^{*}}{2}\check{\kappa}\nabla|\nabla a|^{2}$ and $\mathrm{div}\Big(\big(\tilde{\kappa}_{3}(a)+\bar{\kappa}\big)\nabla
a\otimes\nabla a\Big)$ leads to
\begin{eqnarray}\label{nnn}
\|g_{6}\|^{\ell}_{\tilde{L}_{T}^{1}(\dot{B}^{\frac{d}{q}-3}_{q,\infty})}\lesssim(1+\|(a,m)\|_{E^{p,q}_{T}})\|(a,m)\|^{2}_{E^{p,q}_{T}}.
\end{eqnarray}

Putting those estimates \eqref{E-3.23}-\eqref{nnn} together, we conclude that
\begin{eqnarray}\label{E-3.29}
\|g\|^{\ell}_{\tilde{L}_{T}^{1}(\dot{B}^{\frac{d}{q}-3}_{q,\infty})}\lesssim \big(1+\|(a,m)\|_{E^{p,q}_{T}}\big)\|(a,m)\|^{2}_{E^{p,q}_{T}}.
\end{eqnarray}
Furthermore, combining (\ref{4.13c}), (\ref{h-f}) and (\ref{E-3.29}), we arrive at (\ref{R-E3.1}). Hence, the proof of Proposition \ref{propp} is finished.
\end{proof}


\noindent {\bf{Proof of Theorem \ref{thm2.1} (global existence and uniqueness)}}

Once a priori estimates is constructed, one can prove the global-in-time existence and uniqueness of solutions to (\ref{linearized})-(\ref{composite}) by employing the fixed point argument. It follows from the standard Duhamel formula that
$$\left(
\begin{array}{ccc}
              {a}(t) \\
            {m}(t)\\
           \end{array}
         \right)=\mathcal{G}(t)\left(
\begin{array}{ccc}
{a_{0}}\\
{m_{0}}\\
\end{array}
\right)+\int_{0}^{t}\mathcal G(t-s)\left(
\begin{array}{ccc}
0\\
g(s)\\
\end{array}
         \right)ds,$$
where $\mathcal{G}(t)_{t\geq0}$ is the semi-group associated to the following linear system:
\begin{equation}\label{3.22}
\left\{
\begin{array}{l}\partial_{t}a+\mathrm{div}m=0,\\ [1mm]
 \partial_{t}m-\bar{\mu}\Delta m-\bar{\nu}\nabla\mathrm{div}m-\bar{\kappa}\nabla\Delta a=0.\\[1mm]
 \end{array} \right.
\end{equation}
The explicit definition of $\mathcal{G}(t)=e^{-tA(D)}$ is given by (see e.g., \cite{KT})
\begin{multline*}
\begin{array}{ccc}
\!\!\! e^{-tA(\xi)}\!\!\!\!\!
\end{array}\\
 \triangleq \quad  \!\!\! \left(
\begin{array}{ccc}
\!\!\!\!\!\!\!\!\!\!\!\!\!\!\!\!
\!\!\!\!\!\!\!\!\!\!\!\!\!\!\!\!\!\!\!\!\!\!\!\!\frac{\lambda_{+}(\xi)e^{\lambda_{-}(\xi)t}-\lambda_{-}(\xi)e^{\lambda_{+}(\xi)t}}{\lambda_{+}(\xi)-\lambda_{-}(\xi)}
\,\,\,\,\,\,\,\,\,\,\,\,\,\,\,\,\,\,\,\,\,\,\,\,\,\,\,\,\,\,\,\,\,\,\,\,\,\,\,\,\,\,\,\,\,\,\,i\frac{e^{\lambda_{-}(\xi)t}-e^{\lambda_{+}(\xi)t}}{\lambda_{+}(\xi)-\lambda_{-}(\xi)}\xi^T\\
-i|\xi|^{2}\bar{\kappa}\frac{e^{\lambda_{-}(\xi)t}-e^{\lambda_{+}(\xi)t}}{\lambda_{+}(\xi)-\lambda_{-}(\xi)}\xi
\,\,\,\,\,\,e^{-\bar{\mu}|\xi|^2t}I_{d}\!\!+\!\!\Big(\frac{\lambda_{+}(\xi)e^{\lambda_{+}(\xi)t}-\lambda_{-}(\xi)e^{\lambda_{-}(\xi)t}}{\lambda_{+}(\xi)-\lambda_{-}(\xi)}-e^{-\bar{\mu}|\xi|^2t}\Big)
\frac{\xi\xi^T}{|\xi|^2}\\
\end{array}
\right)\end{multline*}
where
$$\lambda_{\pm}(\xi)=-\frac{\bar{\nu}\pm\sqrt{\bar{\nu}^2-4\bar{\kappa}}}{2}|\xi|^2.$$
Set $$ \left(
\begin{array}{ccc}
{a_{L}}\\
{m_{L}}\\
\end{array}
\right) \triangleq \mathcal{G}(t)\left(
\begin{array}{ccc}
{a_{0}}\\
{m_{0}}\\
\end{array}
\right).$$
Define the functional $\Phi(\tilde{a},\tilde{m})$ in the neighborhood of zero in the space $E^{p,q}_{T}$ by
\begin{equation}\label{4.27}\Phi(\tilde{a},\tilde{m})=\int_{0}^{t}\mathcal{G}(t-s)\left(
\begin{array}{ccc}
0\\
{g(a_{L}+\tilde{a},m_{L}+\tilde{m})}\\
\end{array}
\right)ds.\end{equation}
To get the existence part of Theorem \ref{thm2.1}, it suffices to show that $\Phi(\tilde{a},\tilde{m})$ has a fixed point for in $E^{p,q}_{T}$. For that end,  the proof is divided into two steps: the stability of closed ball $\mathbf{B}(0,r)$ for sufficient small $r$ and the contraction in that ball. Let $a=a_{L}+\tilde{a},m=m_{L}+\tilde{m}$. From Proposition \ref{propp}, we get
\begin{eqnarray}\label{f1}
\|(a_{L},m_{L})\|_{E^{p,q}_{T}}\leq C\|(\nabla a_{0}, m_{0})\|_{\dot{B}^{\frac{d}{p}-1,\frac{d}{q}-3}_{(p,1),(q,\infty)}}\leq C\varepsilon
\end{eqnarray}
and
$$\|\Phi(\tilde{a},\tilde{m})\|_{E^{p,q}_{T}}\leq C(\|g\|^{h}_{\tilde L^{1}_{T}(\dot B^{\frac{d}{p}-1}_{p,1})}+
\|g\|^{\ell}_{\tilde L^{1}_{T}(\dot B^{\frac{d}{q}-3}_{q,\infty})}).$$
Let $r$ be small such that $r\leq\frac{2}{3}R$. Assuming that $2C\varepsilon\leq r$ implies that
$$\|a\|_{\tilde L^{\infty}_{T}(\dot B^{\frac{d}{p}}_{p,1})}\leq C\varepsilon+r\leq R.$$
Thus by calculations in the proof of Proposition \ref{propp}, we have
\begin{eqnarray}&&\|\Phi(\tilde{a},\tilde{m})\|_{E^{p,q}_{T}}\\
\nonumber&\leq&
C_{R}\big(1+\|(a_{L}+\tilde{a},m_{L}+\tilde{m})\|_{E^{p,q}_{T}}\big)\|(a_{L}+\tilde{a},m_{L}+\tilde{m})\|_{E^{p,q}_{T}}^2\\
\nonumber&\leq& C_{R}(1+C\varepsilon+r)(C\varepsilon+r)^2.
\end{eqnarray}
Choosing suitable $(r,\varepsilon)$ such that
\begin{eqnarray}\label{f2}r\leq \min\Big\{1,\frac{1}{30C_{R}}, \frac{2}{3}R\Big\}\,\,\, \mathrm{and} \,\,\,\varepsilon\leq\frac{r}{2C},\end{eqnarray}
we thus have
$$\Phi(\mathbf{B}(0,r))\subset\mathbf{B}(0,r).$$

Next, we prove the contraction property. Let $(\tilde{a}_{1},\tilde{m}_{1})$ and $(\tilde{a}_{2},\tilde{m}_{2})$ be in $\mathbf{B}(0,r)$. We are going to estimate $
\|(\Phi(\tilde{a_{2}},\tilde{m_{2}})-\Phi(\tilde{a_{1}},\tilde{m_{1}}))\|_{E^{p,q}_{T}}$. According to (\ref{4.27}), we get
\begin{eqnarray}\label{contraction}
&&\|(\Phi(\tilde{a_{2}},\tilde{m_{2}})-\Phi(\tilde{a_{1}},\tilde{m_{1}}))\|_{E^{p,q}_{T}}\\
\nonumber&\lesssim&\|g(a_{2},m_{2})-g(a_{1},m_{1})\|^{h}_{\tilde L_{T}^{1}(\dot B^{\frac{d}{p}-1}_{p,1})}+
\|g(a_{2},m_{2})-g(a_{1},m_{1})\|^{\ell}_{\tilde L_{T}^{1}(\dot B^{\frac{d}{q}-3}_{q,\infty})}.
\end{eqnarray}

Bounding those nonlinear terms on the right-hand side of (\ref{contraction}) just follows the similar procedure of a priori estimate in Proposition \ref{propp} combined with Proposition \ref{prop2.26}. Let us for example show the calculation for $g_{4}=a\nabla G(a)$:
\begin{eqnarray*}
&&\|\nabla\big(a_{1} G(a_{1})\big)-\nabla\big(a_{2} G(a_{2})\big)\|^{h}_{\tilde L_{T}^{1}(\dot B^{\frac{d}{p}-1}_{p,1})}\\
&\leq&\|\nabla \big((a_{1}-a_{2})G(a_{1})\big)\|^{h}_{\tilde L_{T}^{1}(\dot B^{\frac{d}{p}-1}_{p,1})}+
\|\nabla\big(a_{2}(G(a_{1})-G(a_{2})\big)\|^{h}_{\tilde L_{T}^{1}(\dot B^{\frac{d}{p}-1}_{p,1})}\\
&\leq&\|\tilde{a}_{1}-\tilde{a}_{2}\|_{\tilde L_{T}^{2}(\dot B^{\frac{d}{p}}_{p,1})}\|G(a_{1})\|_{\tilde L_{T}^{2}(\dot B^{\frac{d}{p}}_{p,1})}+\|a_{2}\|_{\tilde L_{T}^{2}(\dot B^{\frac{d}{p}}_{p,1})}\|G(a_{1})-G(a_{2})\|_{\tilde L_{T}^{2}(\dot B^{\frac{d}{p}}_{p,1})}
\\&\leq&C_{R}(1+\|(a_1,a_2)\|_{E^{p,q}_{T}})\|(a_1,a_2)\|_{E^{p,q}_{T}}\|\tilde{a}_{1}-\tilde{a}_{2}\|_{E^{p,q}_{T}}.
\end{eqnarray*}
On the other hand,
\begin{eqnarray*}
&&\|\nabla\big(a_{1} G(a_{1})\big)-\nabla\big(a_{2} G(a_{2})\big)\|^{\ell}_{\tilde L_{T}^{1}(\dot B^{\frac{d}{q}-3}_{q,\infty})}\\
&\leq&\|\nabla \big((a_{1}-a_{2})G(a_{1})\big)\|^{\ell}_{\tilde L_{T}^{1}(\dot B^{\frac{d}{q}-3}_{q,\infty})}+
\|\nabla\big(a_{2}(G(a_{1})-G(a_{2})\big)\|^{\ell}_{\tilde L_{T}^{1}(\dot B^{\frac{d}{q}-3}_{q,\infty})}
\\&\leq&C_{R}(1+\|(a_1,a_2)\|_{E^{p,q}_{T}})\|(a_1,a_2)\|_{E^{p,q}_{T}}\|\tilde{a}_{1}-\tilde{a}_{2}\|_{E^{p,q}_{T}}.
\end{eqnarray*}
In conclusion, one can get the following error estimate, provided that $r$ and $\varepsilon$ are small enough as in (\ref{f2})
\begin{eqnarray*}
&&\|(\Phi(\tilde{a_{2}},\tilde{m}_{2})-\Phi(\tilde{a_{1}},\tilde{m}_{1}))\|_{E^{p,q}_{T}}
\\&\leq&\Big(1+\|(a_{1},m_{1})\|_{E^{p,q}_{T}}+\|(a_{2},m_{2})\|_{E^{p,q}_{T}}\Big)\Big(\|(a_{1},m_{1})\|_{E^{p,q}_{T}}+\|(a_{2},m_{2})\|_{E^{p,q}_{T}}\Big)\\
&&\times\|(\tilde{a}_{1}-\tilde{a}_{2},\tilde{m}_{1}-\tilde{m}_{2})\|_{E^{p,q}_{T}}
\\&\leq&4C_{R}(C\varepsilon+r+1)(C\varepsilon+r)\|(\tilde{a}_{1}-\tilde{a}_{2},\tilde{m}_{1}-\tilde{m}_{2})\|_{E^{p,q}_{T}}
\\&\leq&\frac{1}{2}\|(\tilde{a}_{1}-\tilde{a}_{2},\tilde{m}_{1}-\tilde{m}_{2})\|_{E^{p,q}_{T}}.
\end{eqnarray*}
Hence, according to the fixed point theorem, there exists a solution $(a,m)$ to (\ref{linearized})-(\ref{composite1}), which obviously belongs to the space $E^{p,q}_{T}$ and
\begin{eqnarray}\label{mima}
\|a\|_{\tilde L^{\infty}_{T}(\dot B^{\frac{d}{p}}_{p,1})}\leq \|(a,m)\|_{E^{p,q}_{T}}\leq \frac{1}{20C_R}
\end{eqnarray}
for any $T>0$, provided that changing $C_R$ into a greater constant if necessary.

Finally, we finish the uniqueness. Let us consider the solutions $(a_{i}, m_{i})$ ($i=1, 2$) constructed in the previous part with the same initial data $(a_{0}, m_{0})$  on $[0,T]$ where $T\in[0,\infty)$. The error variable $(\delta a, \delta m)\triangleq(a_{1}-a_{2}, m_{1}-m_{2})$ satisfies
\begin{equation}
\left\{
\begin{array}{l}\partial_{t}\delta a+\mathrm{div}\delta m=0,\\ [1mm]
 \partial_{t}\delta m-\bar{\mathcal{A}}\delta m-\bar{\kappa}\nabla\Delta \delta a= g(a_{1},m_{1})-g(a_{2},m_{2}),\\[1mm]
(\delta a,\delta m)|_{t=0}=0.\\[1mm]
 \end{array} \right.\label{error-linearized}
\end{equation}
Repeating the linear analysis which is based on the idea of the effective velocity yields
\begin{eqnarray}\label{3.38}
\|(\delta \nabla a,\delta m)\|_{L^{\infty}_{T}(\dot B^{\frac{d}{p}-2}_{p,\infty})}+\|(\delta \nabla a,\delta m)\|_{\tilde L^{1}_{T}(\dot B^{\frac{d}{p}}_{p,\infty})}\leq C\|\delta g\|_{\tilde L^{1}_{T}(\dot B^{\frac{d}{p}-2}_{p,\infty})},
\end{eqnarray}
where $\delta g=g(a_{1},m_{1})-g(a_{2},m_{2})$. For simplicity, we only present the outline of error estimates by using Proposition \ref{prop3.2} and Proposition \ref{prop2.26}. Actually, by  (\ref{mima}) and Proposition \ref{prop2.26}, there holds for $p<2d$
\begin{multline*}
\|\delta g_{1}\|_{\tilde L^{1}_{T}(\dot B^{\frac{d}{p}-2}_{p,\infty})}
\leq C \Big(\big(1+\|a_{i}\|_{\tilde L^{\infty}_{T}(\dot B^{\frac{d}{p}}_{p,1})}\big)\|m_{i}\|_{\tilde L^{2}_{T}(\dot B^{\frac{d}{p}}_{p,1})}\|\delta m\|_{\tilde L^{2}_{T}(\dot B^{\frac{d}{p}-1}_{p,\infty})}\\
+\|\delta a\|_{\tilde L^{\infty}_{T}(\dot B^{\frac{d}{p}-1}_{p,\infty})}\|m_{i}\|^{2}_{\tilde L^{2}_{T}(\dot B^{\frac{d}{p}}_{p,1})}\Big)
\end{multline*}
where we take $s=\frac{d}{p}-1$ in Proposition \ref{prop2.26} which indicates $-d\min(\frac{1}{p},\frac{1}{p'})<s$.

Regarding for $\delta g_{2}$, we obtain
\begin{multline*}
\|\delta g_{2}\|_{\tilde L^{1}_{T}(\dot B^{\frac{d}{p}-2}_{p,\infty})}
\leq C \Big(\|\nabla\delta m\|_{\tilde L^{1}_{T}(\dot B^{\frac{d}{p}-1}_{p,\infty})}\|a_{1}\|_{\tilde L^{\infty}_{T}(\dot B^{\frac{d}{p}}_{p,1})}+
\|\delta m\|_{\tilde L^{2}_{T}(\dot B^{\frac{d}{p}-1}_{p,\infty})}\|\nabla a_{1}\|_{\tilde L^{2}_{T}(\dot B^{\frac{d}{p}}_{p,1})}\\
+\|m_{2}\|_{\tilde L^{2}_{T}(\dot B^{\frac{d}{p}}_{p,1})}\|\nabla \delta a\|_{\tilde L^{2}_{T}(\dot B^{\frac{d}{p}-1}_{p,\infty})}+
\int^{T}_{0}\|\nabla m_{2}\|_{\dot B^{\frac{d}{p}}_{p,1}}\|\delta a\|_{\dot B^{\frac{d}{p}-1}_{p,\infty}}dt\Big).
\end{multline*}
Bounding $\delta g_{3}$ is completely similar to $\delta g_{2}$.  For the error of pressure, we have
\begin{eqnarray*}
\|\delta g_{4}\|_{\tilde L^{1}_{T}(\dot B^{\frac{d}{p}-2}_{p,\infty})}
\leq C\|\delta a\|_{\tilde L^{2}_{T}(\dot B^{\frac{d}{p}}_{p,1})}\|a_{i}\|_{\tilde L^{2}_{T}(\dot B^{\frac{d}{p}-1}_{p,\infty})}, \quad i=1, 2.
\end{eqnarray*}
Also, it is not difficult to handle the term corresponding to $\delta g_{5}$. Precisely,
\begin{eqnarray*}
\|\delta g_{5}\|_{\tilde L^{1}_{T}(\dot B^{\frac{d}{p}-2}_{p,\infty})}
\leq C \big(\|a_{2}\|_{\tilde L^{\infty}_{T}(\dot B^{\frac{d}{p}}_{p,1})}\|\Delta \delta a\|_{\tilde L^{1}_{T}(\dot B^{\frac{d}{p}-1}_{p,\infty})}+\int^{T}_{0}\|\Delta a_{i}\|_{\dot B^{\frac{d}{p}}_{p,1}}\|\delta a\|_{\dot B^{\frac{d}{p}-1}_{p,\infty}}dt\big).
\end{eqnarray*}
Finally, one can estimate $\delta g_{6}$ similarly by repeating above calculations. Thus, by Remark \ref{Rem2.1} and interpolation, we conclude that
\begin{multline}\label{ni}
\|g(a_{1},m_{1})-g(a_{2},m_{2})\|_{\tilde L^{1}_{T}(\dot B^{\frac{d}{p}-2}_{p,\infty})}\leq C\Big(\| a_{i}\|_{\tilde L^{2}_{T}(\dot B^{\frac{d}{p}-1}_{p,1})}\|\nabla \delta a\|_{L^{\infty}_{T}(\dot B^{\frac{d}{p}-2}_{p,\infty})\cap\tilde L^{1}_{T}(\dot B^{\frac{d}{p}}_{p,\infty})}\\+\mathcal{Z}\big(\| a_{i}\|_{\tilde L^{\infty}_{T}(\dot B^{\frac{d}{p}}_{p,1})}+\|(\nabla a_{i},m_{i})\|_{\tilde L^{2}_{T}(\dot B^{\frac{d}{p}}_{p,1})\cap\tilde L^{1}_{T}(\dot B^{\frac{d}{p}+1}_{p,1})}\big)\|\nabla \delta a, \delta m\|_{L^{\infty}_{T}(\dot B^{\frac{d}{p}-2}_{p,\infty})\cap\tilde L^{1}_{T}(\dot B^{\frac{d}{p}}_{p,\infty})}\\
+\int^{T}_{0}\|(\nabla  a_{i},  m_{i})\|_{\dot B^{\frac{d}{p}+1}_{p,1}}\|(\nabla \delta a, \delta m)\|_{\dot B^{\frac{d}{p}-2}_{p,\infty}}dt\Big),
\end{multline}
where $\mathcal{Z}(x)\triangleq x+x^2$. According to the global existence of Theorem \ref{thm2.1}, if taking $\varepsilon>0$ small enough such that
\begin{eqnarray*}\|a_{i}\|_{\tilde L^{\infty}_{T}(\dot B^{\frac{d}{p}}_{p,1})}+\| a_{i}\|_{\tilde L^{2}_{T}(\dot B^{\frac{d}{p}-1}_{p,1})}+\|(\nabla a_{i},m_{i})\|_{\tilde L^{2}_{T}(\dot B^{\frac{d}{p}}_{p,1})\cap\tilde L^{1}_{T}(\dot B^{\frac{d}{p}+1}_{p,1})}\ll1,\end{eqnarray*}
then (\ref{3.38})-(\ref{ni}) implies that
\begin{eqnarray*}
\|(\delta \nabla a,\delta m)\|_{L^{\infty}_{T}(\dot B^{\frac{d}{p}-2}_{p,\infty})}
\leq C\int^{T}_{0}\|(\nabla  a_{i},  m_{i})\|_{\dot B^{\frac{d}{p}+1}_{p,1}}\|(\nabla \delta a, \delta m)\|_{\dot B^{\frac{d}{p}-2}_{p,\infty}}dt.
\end{eqnarray*}
Gronwall's Lemma gives that $(\delta a, \delta m)\equiv0$ immediately in the interval $[0,T]$, since $\|(\nabla  a_{i},  m_{i})\|_{\dot B^{\frac{d}{p}+1}_{p,1}}$ are integrable on $[0,T]$. Then Sobolev embedding and Remark 2.2 ensure that $E^{p,q}_{T}\subset L^{\infty}_{T}(\dot B^{\frac{d}{p}-1}_{p,\infty})\times L^{\infty}_{T}(\dot B^{\frac{d}{p}-2}_{p,\infty})$, which indicates the uniqueness in $E^{p,q}_{T}$. Furthermore, we get the uniqueness in $[0,\infty)$ by the continuity argument.

\subsection{Gevrey analyticity}
In the regime of small data, the spectrum analysis of \eqref{R-E69} implies the parabolic smoothing of solution $(a,m)$, which enables us to establish
Gevrey analyticity for the nonlinear problem (\ref{1.1})-(\ref{Eq:1.2}). The main part of proof lies in the following a priori estimate, which indicates the evolution of
Gevrey regularity.

\begin{prop}\label{propG}
Let $(a,m)$ be the solution constructed in Theorem \ref{thm2.1}. Denote $(A,U)\triangleq (e^{\sqrt {c_{0}t}\Lambda_1}{a},e^{\sqrt {c_{0}t}\Lambda_1}{m})$ where $c_{0}>0$ is some constant. There exists a constant $R_{0}$ that if
$$\|A\|_{\tilde{L}^{\infty}_{T}(\dot{B}^{\frac{d}{p}}_{p,1})}\leq R_{0},$$
then it holds for $q\neq1$ that
\begin{eqnarray}\label{E3.36}
\hspace {5mm}\|(A,U)\|_{E^{p,q}_{T}}\lesssim\|(\nabla a_{0}, m_{0})\|_{\dot{B}^{\frac{d}{p}-1,\frac{d}{q}-3}_{(p,1),(q,\infty)}}+\big(1+\|(A,U)\|_{E^{p,q}_{T}}\big)\|(A,U)\|^{2}_{E^{p,q}_{T}}.\end{eqnarray}
\end{prop}
\begin{proof}
In order to get the Gevrey estimates, one can repeat the procedure leading to Proposition \ref{propp} after employing the Fourier multiplier $e^{\sqrt {c_{0}t}\Lambda_1}$ for some $c_{0}>0$ everywhere. For convenience, let us write $G=e^{\sqrt {c_{0}t}\Lambda_1}{g}$ and begin with the high-frequency estimates.
By using the effective velocity $w=\cQ
m+\alpha \nabla a$ with $\alpha=\frac{1}{2}(\bar{\nu}\pm\sqrt{\bar{\nu}^2-4\bar{\kappa}})$, we have
 \begin{equation*}
\d_tw-\tilde{\alpha}\Delta w=\cQ g,
\end{equation*}
where $\tilde{\alpha}=\bar{\nu}-\alpha$. It follows from the Duhamel's formula that
 \begin{equation*}
w(t)=e^{\tilde{\alpha}t \Delta}w_{0}+\int^{t}_{0}e^{\tilde{\alpha}(t-\tau) \Delta}\cQ g(\tau)d\tau.
\end{equation*}
Define $W\triangleq e^{\sqrt {c_{0}t}\Lambda_1}{w}$. Then $W$ satisfies the following equation
 \begin{equation*}
W(t)=e^{\sqrt{c_{0}t}\Lambda_1+\tilde{\alpha}t \Delta}w_{0}+\int^{t}_{0}
e^{\sqrt{c_{0}(t-\tau)}\Lambda_1+\tilde{\alpha}(t-\tau) \Delta}
e^{\sqrt{c_{0}}(-\sqrt{t-\tau}+\sqrt{t}-\sqrt{\tau})\Lambda_1}\cQ G(\tau)d\tau,
\end{equation*}
where the positive constant $c_0$ is chosen such that $c_{0}\leq2\mathrm{min}\{\tilde{\alpha},\alpha,\bar{\mu}\}$.

Note that the fact that Fourier multipliers can commute each other, so it follows from Lemma \ref{lem5.1} and Lemma \ref{lem5.2} that (restricted in high
frequency cut-off with $j\geq j_{0}$)
\begin{equation}\label{Cal}
\|W\|^{h}_{\tilde{L}_{T}^{\infty}(\dot{B}^{\frac{d}{p}-1}_{p,1})}+
\|W\|^{h}_{\tilde{L}_{T}^{1}(\dot{B}^{\frac{d}{p}+1}_{p,1})}\lesssim  \|w_0\|^{h}_{\dot{B}^{\frac{d}{p}-1}_{p,1}}+\|\cQ
G\|^{h}_{\tilde{L}_{T}^{1}(\dot{B}^{\frac{d}{p}-1}_{p,1})}.
\end{equation}
We recall the momentum $\mathcal{Q}m$ can be written by
\begin{equation}\label{Qm}
\d_t\cQ m-\alpha \Delta \cQ m=\frac{\bar{\kappa}}{\alpha}\Delta w+\cQ g.
\end{equation}
Applying $e^{\sqrt {c_{0}t}\Lambda_1}$ to (\ref{Qm}) and repeating the same calculation as (\ref{Cal}) imply that
\begin{equation*}
\|\mathcal{Q}U\|^{h}_{\tilde{L}_{T}^{\infty}(\dot{B}^{\frac{d}{p}-1}_{p,1})}+
\|\mathcal{Q}U\|^{h}_{\tilde{L}_{T}^{1}(\dot{B}^{\frac{d}{p}+1}_{p,1})}\lesssim
\|\mathcal{Q}m_0\|^{h}_{\dot{B}^{\frac{d}{p}-1}_{p,1}}+\|w_{0}\|^{h}_{\dot{B}^{\frac{d}{p}-1}_{p,1}}
+\|\cQ G\|^{h}_{\tilde{L}_{T}^{1}(\dot{B}^{\frac{d}{p}-1}_{p,1})}.
\end{equation*}
Therefore, by (\ref{4.4}), we deduct that
\begin{equation}\label{Gh2}
\|(\nabla A,\mathcal{Q}U)\|^{h}_{\tilde{L}_{T}^{\infty}(\dot{B}^{\frac{d}{p}-1}_{p,1})}+\|(\nabla A,\mathcal{Q}U)\|^{h}_{\tilde{L}_{T}^{1}(\dot{B}^{\frac{d}{p}+1}_{p,1})}\lesssim
\|(\nabla a_{0},\mathcal{Q}m_{0})\|^{h}_{\dot{B}^{\frac{d}{p}-1}_{p,1}}+
\|\mathcal{Q}G\|^{h}_{\tilde{L}_{T}^{1}(\dot{B}^{\frac{d}{p}-1}_{p,1})}.
\end{equation}
Regarding the incompressible part, applying $e^{\sqrt {c_{0}t}\Lambda_1}$ to (\ref{incompressible}) yields
\begin{equation}\label{4.8r}
\|\mathcal{P}U\|^{h}_{\tilde{L}_{T}^{\infty}(\dot{B}^{\frac{d}{p}-1}_{p,1})}+\|\mathcal{P}U\|^{h}_{\tilde{L}_{T}^{1}(\dot{B}^{\frac{d}{p}+1}_{p,1})}\lesssim
\|\mathcal{P}m_{0}\|^{h}_{\dot{B}^{\frac{d}{p}-1}_{p,1}}+
\|\mathcal{P}G\|^{h}_{\tilde{L}_{T}^{1}(\dot{B}^{\frac{d}{p}-1}_{p,1})}.
\end{equation}
Consequently, it follows from (\ref{Gh2}) and (\ref{4.8r}) that
\begin{equation}\label{Gh}
\|(\nabla A,U)\|^{h}_{\tilde{L}_{T}^{\infty}(\dot{B}^{\frac{d}{p}-1}_{p,1})}+\|(\nabla A,U)\|^{h}_{\tilde{L}_{T}^{1}(\dot{B}^{\frac{d}{p}+1}_{p,1})}\lesssim
\|(\nabla a_{0},m_{0})\|^{h}_{\dot{B}^{\frac{d}{p}-1}_{p,1}}+
\|G\|^{h}_{\tilde{L}_{T}^{1}(\dot{B}^{\frac{d}{p}-1}_{p,1})}.
\end{equation}
In spirit of Gevrey estimates in high frequencies, we obtain
\begin{equation}\label{Gl}
\|(\nabla A,U)\|^{\ell}_{\tilde{L}_{T}^{\infty}(\dot{B}^{\frac{d}{q}-3}_{q,\infty})}+\|(\nabla A,U)\|^{\ell}_{\tilde{L}_{T}^{1}(\dot{B}^{\frac{d}{q}-1}_{q,\infty})}\lesssim
\|(\nabla a_{0},m_{0})\|^{\ell}_{\dot{B}^{\frac{d}{q}-3}_{q,\infty}}+
\|G\|^{\ell}_{\tilde{L}_{T}^{1}(\dot{B}^{\frac{d}{q}-3}_{q,\infty})}.
\end{equation}
Together (\ref{Gh}) with (\ref{Gl}), we conclude that
\begin{equation}\label{lin}
\|(\nabla A,U)\|_{E^{p,q}_{T}}\lesssim\|(\nabla a_{0},m_{0})\|_{\dot{B}^{\frac{d}{p}-1,\frac{d}{q}-3}_{(p,1),(q,\infty)}}+
\|G\|^{h}_{\tilde{L}_{T}^{1}(\dot{B}^{\frac{d}{p}-1}_{p,1})}+
\|G\|^{\ell}_{\tilde{L}_{T}^{1}(\dot{B}^{\frac{d}{q}-3}_{q,\infty})}.
\end{equation}

What left is the estimate of $G$. Let us handle the norm $\|G\|^{h}_{\tilde{L}_{T}^{1}(\dot{B}^{\frac{d}{p}-1}_{p,1})}$ first. Setting $R_0\triangleq\mathrm\min\{R_{Q}, R_{G}, R_{\tilde{\lambda}}, R_{\tilde{\mu}}, R_{\tilde{\kappa}_{i}}\}$ with $i=1, 2, 3$,  which stands for the minimum of convergence radii of those analytic functions appearing in the nonlinear term $g$. It follows from Propositions \ref{prop5.1}-\ref{prop5.2} that
\begin{eqnarray}\label{hip}
\|G_{1}\|^{h}_{\tilde{L}_{T}^{1}(\dot{B}^{\frac{d}{p}-1}_{p,1})}&\lesssim&\|e^{\sqrt {c_{0}t}\Lambda_1}\Big(\big(Q(a)-1\big)m\otimes
m\Big)\|^{h}_{\tilde{L}_{T}^{1}(\dot{B}^{\frac{d}{p}}_{p,1})}\\
\nonumber&\lesssim&(1+\|A\|_{\tilde{L}^{\infty}(\dot{B}^{\frac{d}{p}}_{p,1})})\|U\|^{2}_{\tilde{L}_{T}^{2}(\dot{B}^{\frac{d}{p}}_{p,1})}\\
\nonumber&\lesssim&(1+\|(A,U)\|_{E^{p,q}_{T}})\|(A,U)\|^{2}_{E^{p,q}_{T}}.
\end{eqnarray}
For $G_{2}$, without loss of generality, it suffices to deal with the term $\bar{\mu}\Delta(Q(a)m)$. Precisely,
\begin{eqnarray*}
\|\bar{\mu}e^{\sqrt {c_{0}t}\Lambda_1}\Delta(Q(a)m)\|^{h}_{\tilde{L}^{1}_{T}(\dot{B}^{\frac{d}{p}-1}_{p,1})}&\lesssim&\|e^{\sqrt {c_{0}t}\Lambda_1}Q(a)\nabla m\|
^{h}_{\tilde{L}^{1}_{T}(\dot{B}^{\frac{d}{p}}_{p,1})}+\|e^{\sqrt {c_{0}t}\Lambda_1}\nabla Q(a)m\|
^{h}_{\tilde{L}^{1}_{T}(\dot{B}^{\frac{d}{p}}_{p,1})}\\
\nonumber&\triangleq&I_{1}+I_{2},
\end{eqnarray*}
where
\begin{eqnarray*}
I_{1}\lesssim\|A\|_{\tilde{L}^{\infty}_{T}(\dot{B}^{\frac{d}{p}}_{p,1})}\|U\|_{\tilde{L}^{1}_{T}(\dot{B}^{\frac{d}{p}+1}_{p,1})}.
\end{eqnarray*}
For $I_{2}$, it follows from Proposition \ref{prop5.1} that
\begin{eqnarray*}
I_{2}\lesssim\|e^{\sqrt {c_{0}t}\Lambda_1}Q(a)\|_{\tilde{L}^{2}_{T}(\dot{B}^{\frac{d}{p}+1}_{p,1})}\|U\|_{\tilde{L}^{2}_{T}(\dot{B}^{\frac{d}{p}}_{p,1})}.
\end{eqnarray*}
Here, we cannot apply Proposition \ref{prop5.2} to estimate $\|e^{\sqrt {c_{0}t}\Lambda_1}Q(a)\|_{\tilde{L}^{2}_{T}\dot{B}^{\frac{d}{p}+1}_{p,1}}$ directly,
since the regularity index is greater than $d/p$. To do this, we develop Proposition \ref{prop4.1} where the regularity index can be relaxed to $s>d/p$.
The proof of Proposition \ref{prop4.1} will be given in Appendix for clarity. Consequently, we have
\begin{eqnarray*}
I_{2}\lesssim\|A\|_{\tilde{L}^{2}_{T}(\dot{B}^{\frac{d}{p}+1}_{p,1})}\|U\|_{\tilde{L}^{2}_{T}(\dot{B}^{\frac{d}{p}}_{p,1})}.
\end{eqnarray*}
Thus one can get
\begin{eqnarray}
&&\|\bar{\mu}e^{\sqrt {c_{0}t}\Lambda_1}\Delta(Q(a)m)\|^{h}_{\tilde{L}^{1}_{T}(\dot{B}^{\frac{d}{p}-1}_{p,1})}\nonumber\\&\lesssim&\|A\|_{\tilde{L}^{2}_{T}(\dot{B}^{\frac{d}{p}+1}_{p,1})}\|U\|_{\tilde{L}^{2}_{T}(\dot{B}^{\frac{d}{p}}_{p,1})}+
\|A\|_{\tilde{L}^{\infty}_{T}(\dot{B}^{\frac{d}{p}}_{p,1})}\|U\|_{\tilde{L}^{1}_{T}(\dot{B}^{\frac{d}{p}+1}_{p,1})}.
\end{eqnarray}
Bounding the term $(\bar{\mu}+\bar{\lambda})\nabla\mathrm{div}(Q(a)m)$ is similar, so we get
\begin{eqnarray}
&&\|G_{2}\|^{h}_{\tilde{L}^{1}_{T}(\dot{B}^{\frac{d}{p}-1}_{p,1})}\nonumber\\&\lesssim&\hspace{-3mm}\|A\|_{\tilde{L}^{2}_{T}(\dot{B}^{\frac{d}{p}+1}_{p,1})}\|U\|_{\tilde{L}^{2}_{T}(\dot{B}^{\frac{d}{p}}_{p,1})}+
\|A\|_{\tilde{L}^{\infty}_{T}(\dot{B}^{\frac{d}{p}}_{p,1})}\|U\|_{\tilde{L}^{1}_{T}(\dot{B}^{\frac{d}{p}+1}_{p,1})}\lesssim\|(A,U)\|^{2}_{E^{p,q}_{T}}.
\end{eqnarray}
Also, it follows from Proposition \ref{prop4.1} that
\begin{eqnarray}
&&\|G_{3}\|^{h}_{\tilde{L}^{1}_{T}(\dot{B}^{\frac{d}{p}-1}_{p,1})}\\
\nonumber&\lesssim&\|e^{\sqrt
{c_{0}t}\Lambda_1}\big\{\tilde{\mu}(a)D\big((1-Q(a))m\big)\big\}\|^{h}_{\tilde{L}^{1}_{T}(\dot{B}^{\frac{d}{p}}_{p,1})}\\ \nonumber&&+\|e^{\sqrt {c_{0}t}\Lambda_1}\big\{\tilde{\Lambda_1}(a)\mathrm{div}\big((1-Q(a))m\big)\big\}\|^{h}_{\tilde{L}^{1}_{T}(\dot{B}^{\frac{d}{p}}_{p,1})}\\
\nonumber&\lesssim&\|A\|_{\tilde{L}^{\infty}_{T}(\dot{B}^{\frac{d}{p}}_{p,1})}\Big(\|A\|_{\tilde{L}^{2}_{T}(\dot{B}^{\frac{d}{p}+1}_{p,1})}\|U\|_{\tilde{L}^{2}_{T}(\dot{B}^{\frac{d}{p}}_{p,1})}
+(1+\|A\|_{\tilde{L}^{\infty}_{T}(\dot{B}^{\frac{d}{p}}_{p,1})})\|U\|_{\tilde{L}^{1}_{T}(\dot{B}^{\frac{d}{p}+1}_{p,1})})\Big)\\
\nonumber&\lesssim&\big(1+\|(A,U)\|_{E^{p,q}_{T}}\big)\|(A,U)\|^{2}_{E^{p,q}_{T}}.
\end{eqnarray}
The nonlinear pressure term can be estimated as follows
\begin{equation}
\|G_{4}\|^{h}_{\tilde{L}^{1}_{T}(\dot{B}^{\frac{d}{p}-1}_{p,1})}\lesssim\|e^{\sqrt
{c_{0}t}\Lambda_1}\big(aG(a)\big)\|^{h}_{\tilde{L}^{1}_{T}(\dot{B}^{\frac{d}{p}}_{p,1})}\lesssim
\|A\|^{2}_{\tilde{L}^{2}_{T}(\dot{B}^{\frac{d}{p}}_{p,1})}\lesssim\|(A,U)\|^{2}_{E^{p,q}_{T}}.
\end{equation}
Regarding the high frequencies of those capillary terms, it follows that
\begin{multline}\label{hop}
\|G_{5}\|^{h}_{\tilde{L}^{1}_{T}(\dot{B}^{\frac{d}{p}-1}_{p,1})}\lesssim\|e^{\sqrt {c_{0}t}\Lambda_1}\big(\tilde{\kappa}_{1}(a)\Delta
a\big)\|^{h}_{\tilde{L}^{1}_{T}(\dot{B}^{\frac{d}{p}}_{p,1})}\\ \lesssim
\|A\|_{\tilde{L}^{\infty}_{T}(\dot{B}^{\frac{d}{p}}_{p,1})}{\|A\|_{\tilde{L}^{1}_{T}(\dot{B}^{\frac{d}{p}+2}_{p,1})}}
\lesssim\|(A,U)\|^{2}_{E^{p,q}_{T}}
\end{multline}
and
\begin{eqnarray}\label{kow}
&&\|G_{6}\|^{h}_{\tilde{L}^{1}_{T}(\dot{B}^{\frac{d}{p}-1}_{p,1})}\\
\nonumber&\lesssim&\|e^{\sqrt {c_{0}t}\Lambda_1}\big\{\big(\tilde{\kappa}_{2}(a)+\check{\kappa}\big)|\nabla a|^{2}\big\}\|^{h}_{\tilde{L}^{1}_{T}(\dot{B}^{\frac{d}{p}}_{p,1})}+
\|e^{\sqrt {c_{0}t}\Lambda_1}\big\{\big(\tilde{\kappa}_{3}(a)+\bar{\kappa}\big)\nabla a\otimes\nabla a\big\}\|^{h}_{\tilde{L}^{1}_{T}(\dot{B}^{\frac{d}{p}}_{p,1})}\\
\nonumber&\lesssim&(1+\|A\|_{\tilde{L}^{\infty}_{T}(\dot{B}^{\frac{d}{p}}_{p,1})})\|A\|^{2}_{\tilde{L}^{2}_{T}(\dot{B}^{\frac{d}{p}+1}_{p,1})}\\
\nonumber&\lesssim&(1+\|(A,U)\|_{E^{p,q}_{T}}\big)\|(A,U)\|^{2}_{E^{p,q}_{T}}.
\end{eqnarray}
Adding those estimates (\ref{hip})-(\ref{kow}) together,  we conclude that
\begin{eqnarray}\label{hf non}
\|G\|^{h}_{\tilde{L}^{1}_{T}(\dot{B}^{\frac{d}{p}-1}_{p,1})}\lesssim
\big(1+\|(A,U)\|_{E^{p,q}_{T}}\big)\|(A,U)\|^{2}_{E^{p,q}_{T}}.
\end{eqnarray}

According to Lemma \ref{lem key G} (taking $s=\frac{d}{p}$, $k_{1}=1$ and $k=2$), we can establish new product estimates where Gevrey regularity is involved.
\begin{prop}\label{cor2.05}
Let $1 < q\leq p\leq 2q$, $p<d$ and
$$\frac{1}{q}<\frac{1}{p}+\frac{2}{d}.$$
Then it holds that
$$\|e^{\sqrt {c_{0}t}\Lambda_1}(ab)\|_{\dot{B}^{\frac{d}{q}-2}_{q,\infty}}\lesssim\|A\|_{\dot{B}^{\frac{d}{p}-2}_{p,\infty}}\|B\|_{\dot{B}^{\frac{d}{p}}_{p,\infty}}+
\|B\|_{\dot{B}^{\frac{d}{p}-2}_{p,\infty}}\|A\|_{\dot{B}^{\frac{d}{p}}_{p,\infty}}.$$
\end{prop}

As the direct consequence, we have also the following version in the Chemin-Lerner spaces.
\begin{cor}\label{cor2.1}
Set $\rho,\rho_{1},\rho_{2},\bar{\rho}_{1},\bar{\rho}_{2}\in[1,\infty]$. Let $1 < q\leq p\leq 2q$, $p<d$ and
$$\frac{1}{q}<\frac{1}{p}+\frac{2}{d}.$$
Then it holds that
$$\|e^{\sqrt {c_{0}t}\Lambda_1}(ab)\|_{\tilde{L}_{T}^{\rho}(\dot{B}^{\frac{d}{q}-2}_{q,\infty})}\lesssim\|A\|_{\tilde{L}_{T}^{\rho_{1}}(\dot{B}^{\frac{d}{p}-2}_{p,\infty})}
\|B\|_{\tilde{L}_{T}^{\rho_{2}}(\dot{B}^{\frac{d}{p}}_{p,\infty})}+\|B\|_{\tilde{L}_{T}^{\bar{\rho}_{1}}(\dot{B}^{\frac{d}{p}-2}_{p,\infty})}
\|A\|_{\tilde{L}_{T}^{\bar{\rho}_{2}}(\dot{B}^{\frac{d}{p}}_{p,\infty})}$$
for any $T>0$, where $\frac{1}{\rho}=\frac{1}{\rho_{1}}+\frac{1}{\rho_{2}}=\frac{1}{\bar{\rho}_{1}}+\frac{1}{\bar{\rho}_{2}}$.
\end{cor}

Indeed, Corollary \ref{cor2.1} and Proposition \ref{prop4.1} are sufficient to estimate the low-frequency norm $\|G\|^{\ell}_{\tilde{L}_{T}^{1}(\dot{B}^{\frac{d}{q}-3}_{q,\infty})}$. Precisely,
\begin{eqnarray}\label{kjk}
\|G_{1}\|^{\ell}_{\tilde{L}_{T}^{1}(\dot{B}^{\frac{d}{q}-3}_{q,\infty})}&\lesssim&\|e^{\sqrt {c_{0}t}\Lambda_1}\big\{\big(Q(a)-1\big)(m\otimes
m)\big\}\|^{\ell}_{\tilde{L}_{T}^{1}(\dot{B}^{\frac{d}{q}-2}_{q,\infty})}\\
\nonumber&\lesssim&\big(1+\|A\|_{\tilde{L}_{T}^{\infty}(\dot{B}^{\frac{d}{p}}_{p,1})}\big)\|U\|_{\tilde{L}_{T}^{\infty}(\dot{B}^{\frac{d}{p}-2}_{p,\infty})}
\|U\|_{\tilde{L}_{T}^{1}(\dot{B}^{\frac{d}{p}}_{p,1})}\\
\nonumber&\lesssim&(1+\|(A,U)\|_{E^{p,q}_{T}})\|(A,U)\|^{2}_{E^{p,q}_{T}}
\end{eqnarray}
and
\begin{eqnarray}
&&\|G_{2}\|^{\ell}_{\tilde{L}_{T}^{1}(\dot{B}^{\frac{d}{q}-3}_{q,\infty})}\\
\nonumber&\lesssim&\|e^{\sqrt {c_{0}t}\Lambda_1}\nabla\big(Q(a)m\big)\|^{\ell}_{\tilde{L}_{T}^{1}(\dot{B}^{\frac{d}{q}-2}_{q,\infty})}+
\|e^{\sqrt {c_{0}t}\Lambda_1}\mathrm{div}\big(Q(a)m\big)\|^{\ell}_{\tilde{L}_{T}^{1}(\dot{B}^{\frac{d}{q}-2}_{q,\infty})}\\
\nonumber&\lesssim&\|\nabla A\|_{\tilde{L}_{T}^{\infty}(\dot{B}^{\frac{d}{p}-2}_{p,\infty})}\|U\|_{\tilde{L}_{T}^{1}(\dot{B}^{\frac{d}{p}}_{p,\infty})}
+\|U\|_{\tilde{L}_{T}^{\infty}(\dot{B}^{\frac{d}{p}-2}_{p,\infty})}
\|\nabla A\|_{\tilde{L}_{T}^{1}(\dot{B}^{\frac{d}{p}}_{p,1})}\\
\nonumber&+&\|A\|_{\tilde{L}_{T}^{\infty}(\dot{B}^{\frac{d}{p}-2}_{p,\infty})}\|\nabla U\|_{\tilde{L}_{T}^{1}(\dot{B}^{\frac{d}{p}}_{p,\infty})}
+\|\nabla U\|_{\tilde{L}_{T}^{2}(\dot{B}^{\frac{d}{p}-2}_{p,\infty})}
\|A\|_{\tilde{L}_{T}^{2}(\dot{B}^{\frac{d}{p}}_{p,1})}\\
\nonumber&\lesssim&\|(A,U)\|^{2}_{E^{p,q}_{T}}.
\end{eqnarray}

For brevity, we just pay attention to the term $\mathrm{div}\Big(\tilde{\mu}(a)D\big(Q(a)m\big)\Big)$ in $G_{3}$. By Corollary \ref{cor2.1}, we have
\begin{eqnarray}
&&\|e^{\sqrt {c_{0}t}\Lambda_1}\mathrm{div}\Big(\tilde{\mu}(a)D\big(Q(a)m\big)\Big)\|^{\ell}_{\tilde{L}^{1}_{T}(\dot{B}^{\frac{d}{q}-3}_{q,\infty})}\\
\nonumber&\lesssim&
\|e^{\sqrt {c_{0}t}\Lambda_1}\tilde{\mu}(a)\|_{\tilde{L}_{T}^{\infty}(\dot{B}^{\frac{d}{p}-2}_{p,\infty})}\|e^{\sqrt
{c_{0}t}\Lambda_1}D\big(Q(a)m\big)\|_{\tilde{L}_{T}^{1}(\dot{B}^{\frac{d}{p}}_{p,\infty})}\\
\nonumber&+&
\|e^{\sqrt {c_{0}t}\Lambda_1}\tilde{\mu}(a)\|_{\tilde{L}_{T}^{1}(\dot{B}^{\frac{d}{p}}_{p,\infty})}\|e^{\sqrt
{c_{0}t}\Lambda_1}D\big(Q(a)m\big)\|_{\tilde{L}_{T}^{\infty}(\dot{B}^{\frac{d}{p}-2}_{p,\infty})}\\
\nonumber&\lesssim&\|A\|_{_{E^{p,q}_{T}}}\big(
\|\nabla A\|_{\tilde{L}_{T}^{2}(\dot{B}^{\frac{d}{p}}_{p,1})}\|U\|_{\tilde{L}_{T}^{2}(\dot{B}^{\frac{d}{p}}_{p,1})}+
\|A\|_{\tilde{L}_{T}^{\infty}(\dot{B}^{\frac{d}{p}}_{p,1})}\|\nabla U\|_{\tilde{L}_{T}^{1}(\dot{B}^{\frac{d}{p}}_{p,1})}\big).
\end{eqnarray}
Repeating above calculations to another term in $G_{3}$, we can get
\begin{eqnarray}
\|G_{3}\|^{\ell}_{\tilde{L}_{T}^{1}(\dot{B}^{\frac{d}{q}-3}_{q,\infty})}
\lesssim(1+\|(A,U)\|_{E^{p,q}_{T}})\|(A,U)\|^{2}_{E^{p,q}_{T}}.
\end{eqnarray}
It follows that
\begin{multline}
\|G_{4}\|^{\ell}_{\tilde{L}_{T}^{1}(\dot{B}^{\frac{d}{q}-3}_{q,\infty})}
\lesssim\|e^{\sqrt {c_{0}t}\Lambda_1}\big(aG(a)\big)\|^{\ell}_{\tilde{L}_{T}^{1}(\dot{B}^{\frac{d}{q}-2}_{q,\infty})}
\\ \lesssim\|A\|_{\tilde{L}_{T}^{\infty}(\dot{B}^{\frac{d}{p}-2}_{p,\infty})}\big(\|A\|_{\tilde{L}_{T}^{1}(\dot{B}^{\frac{d}{p}}_{p,\infty})}
 +
\|A\|^{2}_{\tilde{L}_{T}^{2}(\dot{B}^{\frac{d}{p}}_{p,1})}\big)
\lesssim(1+\|(A,U)\|_{E^{p,q}_{T}})\|(A,U)\|^{2}_{E^{p,q}_{T}}.
\end{multline}
Regarding the Korteweg terms in low frequencies, we have
\begin{multline}
\|G_{5}\|^{\ell}_{\tilde{L}^{1}_{T}(\dot{B}^{\frac{d}{q}-3}_{q,\infty})}\lesssim\|e^{\sqrt {c_{0}t}\Lambda_1}\big(\tilde{\kappa}_{1}(a)\Delta
a\big)\|^{\ell}_{\tilde{L}^{1}_{T}(\dot{B}^{\frac{d}{q}-2}_{q,\infty})}\\
\lesssim\|\Delta A\|_{\tilde{L}^{\infty}_{T}(\dot{B}^{\frac{d}{p}-2}_{p,\infty})}
\|e^{\sqrt {c_{0}t}\Lambda_1}\tilde{\kappa}_{1}(a)\|_{\tilde{L}^{1}_{T}(\dot{B}^{\frac{d}{p}}_{p,\infty})}+\|\Delta A\|_{\tilde{L}^{1}_{T}(\dot{B}^{\frac{d}{p}}_{p,\infty})}
\|A\|_{\tilde{L}^{\infty}_{T}(\dot{B}^{\frac{d}{p}-2}_{p,\infty})}
\\ \lesssim(1+\|(A,U)\|_{E^{p,q}_{T}})\|(A,U)\|^{2}_{E^{p,q}_{T}}
\end{multline}
and
\begin{eqnarray}\label{jkj}
\|G_{6}\|^{\ell}_{\tilde{L}^{1}_{T}(\dot{B}^{\frac{d}{q}-3}_{q,\infty})}&\lesssim&\|e^{\sqrt {c_{0}t}\Lambda_1}\big\{\big(\tilde{\kappa}_{2}(a)+\check{\kappa}\big)|\nabla
a|^{2}\big\}\|^{\ell}_{\tilde{L}^{1}_{T}(\dot{B}^{\frac{d}{q}-2}_{q,\infty})}\\
\nonumber&&+
\|e^{\sqrt {c_{0}t}\Lambda_1}\big\{\big(\tilde{\kappa}_{3}(a)+\bar{\kappa}\big)\nabla a\otimes\nabla a\big\}\|^{\ell}_{\tilde{L}^{1}_{T}(\dot{B}^{\frac{d}{q}-2}_{q,\infty})}\\
\nonumber&\lesssim&(1+\|A\|_{\tilde{L}^{\infty}_{T}(\dot{B}^{\frac{d}{p}}_{p,1})})
\|A\|^{2}_{\tilde{L}^{2}_{T}(\dot{B}^{\frac{d}{p}}_{p,1})}\\
\nonumber&\lesssim&(1+\|(A,U)\|_{E^{p,q}_{T}})\|(A,U)\|^{2}_{E^{p,q}_{T}}.
\end{eqnarray}

Combining (\ref{kjk})-(\ref{jkj}), we are eventually led to the low-frequency estimate
\begin{eqnarray}\label{mml}
\|G\|^{\ell}_{\tilde{L}^{1}_{T}(\dot{B}^{\frac{d}{q}-3}_{q,\infty})}\lesssim(1+\|(A,U)\|_{E^{p,q}_{T}})\|(A,U)\|^{2}_{E^{p,q}_{T}}.
\end{eqnarray}
Hence, \eqref{E3.36} is followed by adding those inequalities \eqref{lin},\eqref{hf non} and \eqref{mml}, which completes the proof of Proposition \ref{propG}.
\end{proof}

With the help of Proposition \ref{propG}, it is not difficult to work out a similar fixed point argument and to prove the Gevrey regularity part of Theorem \ref{thm2.1}.

\section{Appendix: nonlinear estimates in Besov (-Gevrey) spaces}\setcounter{equation}{0}\label{appendix}
In the last section, we present new estimates for paraproduct and remainder, which can be regarded as the generalization of those efforts in \cite{D2}.
\begin{lem}\label{lem key}
Let $s,k,k_{1},k_{2}\in \mathbb{R}$ and $1 \leq p,q\leq\infty$. Then there holds for $k=k_{1}+k_{2}$ and $1=1/p+1/p'$ such that
\begin{itemize}
\item[$(i)$] $\|T_{a}b\|_{\dot{B}^{s-k+\frac{d}{q}-\frac{d}{p}}_{q,\infty}}\lesssim\|a\|_{\dot{B}^{\frac{d}{p}-k_{1}}_{p,\infty}}\|b\|_{\dot{B}^{s-k_{2}}_{p,\infty}} \,\,\,\,\mathrm{if}\,\,\,\,
    k_{1}> d\max(0,\frac{1}{q}-\frac{1}{p})\, \mbox{and} \, \, p\leq2q;$
\item [$(ii)$] $\|R(a,b)\|_{\dot{B}^{s-k+\frac{d}{q}-\frac{d}{p}}_{q,\infty}}\lesssim\|a\|_{\dot{B}^{\frac{d}{p}-k_{1}}_{p,\infty}}\|b\|_{\dot{B}^{s-k_{2}}_{p,\infty}}\,\,\,\,\mathrm{if}\,\,\,\,
s>k-d\min(\frac{1}{p},\frac{1}{p'})\, \mbox{and} \, \,  p\leq2q.$
\end{itemize}
\end{lem}
\begin{proof}
To prove $\mathrm(i)$, it firstly follows from the definition of $T_{a}b$ and the spectral cut-off property that
$$\dot{\Delta}_{j}T_{a}b=\dot{\Delta}_{j}\Big(\sum_{j'}\dot{S}_{j'-1}a\dot{\Delta}_{j'}b\Big)=\sum_{|j-j'|\leq 4}\dot{\Delta}_{j}(\dot{S}_{j'-1}a\dot{\Delta}_{j'}b).$$

On the one hand, we assume that $q\leq p$. Set $\frac{1}{q}=\frac{1}{m}+\frac{1}{p}$. The assumption $p\leq2q$ implies that $p\leq m$. Hence, we have
\begin{eqnarray*}
\|\dot{\Delta}_{j}T_{a}b\|_{L^q}&\lesssim&\sum_{|j-j'|\leq 4}\|\dot{S}_{j'-1}a\dot{\Delta}_{j'}b\|_{L^q}\\
&\lesssim&\sum_{|j-j'|\leq 4}\sum_{l\leq j'-2}\|\dot{\Delta}_{l}a\|_{L^m}\|\dot{\Delta}_{j'}b\|_{L^p}\\
&\lesssim&\sum_{|j-j'|\leq 4}\big(\sum_{l\leq j'-2}2^{(\frac{d}{p}-\frac{d}{q}+k_{1})l}2^{(\frac{d}{p}-k_{1})l}\|\dot{\Delta}_{l}a\|_{L^p}\big)\|\dot{\Delta}_{j'}b\|_{L^p},
\end{eqnarray*}
Note that $\frac{d}{p}-\frac{d}{q}+k_{1}> 0$, we deduct that
\begin{eqnarray*}\label{qqr}
\|\dot{\Delta}_{j}T_{a}b\|_{L^q}
&\lesssim&\sum_{|j-j'|\leq 4}\big(\sum_{l\leq j'-2}2^{(\frac{d}{p}-\frac{d}{q}+k_{1})l}\big)\|a\|_{\dot{B}^{\frac{d}{p}-k_{1}}_{p,\infty}}\|\dot{\Delta}_{j'}b\|_{L^p}\\
\nonumber&\lesssim&\sum_{|j-j'|\leq 4}2^{(\frac{d}{p}-\frac{d}{q}+k_{1})j'}\|\dot{\Delta}_{j'}b\|_{L^p}\|a\|_{\dot{B}^{\frac{d}{p}-k_{1}}_{p,\infty}}.
\end{eqnarray*}
On the other hand, if $q>p$ then
\begin{eqnarray*}
\|\dot{\Delta}_{j}T_{a}b\|_{L^q}&\lesssim&\sum_{|j-j'|\leq 4}\|\dot{S}_{j'-1}a\dot{\Delta}_{j'}b\|_{L^q}\\
&\lesssim&\sum_{|j-j'|\leq 4}\sum_{l\leq j'-2}\|\dot{\Delta}_{l}a\|_{L^\infty}\|\dot{\Delta}_{j'}b\|_{L^q}\\
&\lesssim&\sum_{|j-j'|\leq 4}2^{(\frac{d}{p}-\frac{d}{q})j'}\big(\sum_{l\leq j'-2}2^{k_{1}l}\big)\|a\|_{\dot{B}^{\frac{d}{p}-k_{1}}_{p,\infty}}\|\dot{\Delta}_{j'}b\|_{L^p}\\
&\lesssim&\sum_{|j-j'|\leq 4}2^{(\frac{d}{p}-\frac{d}{q}+k_{1})j'}\|\dot{\Delta}_{j'}b\|_{L^p}\|a\|_{\dot{B}^{\frac{d}{p}-k_{1}}_{p,\infty}},
\end{eqnarray*}
where $k_{1}> 0$ was used. Consequently, employing Young's inequality enables us to get
\begin{eqnarray*}
\|T_{a}b\|_{\dot{B}^{s-k+\frac{d}{q}-\frac{d}{p}}_{q,\infty}}
\lesssim\|a\|_{\dot{B}^{\frac{d}{p}-k_{1}}_{p,\infty}} \|b\|_{\dot{B}^{s-k_{2}}_{p,\infty}}.
\end{eqnarray*}

We turn to prove $(ii)$. By the spectrum cut-off, one has
$$\dot{\Delta}_{j}R(a,b)=\dot{\Delta}_{j}\Big(\sum_{j'}\tilde{\dot{\Delta}}_{j'}a\dot{\Delta}_{j'}b\Big)=\sum_{j\leq j'+2}\dot{\Delta}_{j}(\tilde{\dot{\Delta}}_{j'}a\dot{\Delta}_{j'}b).$$
We consider the case $1\leq p\leq2$ first. By H\"{o}lder and Bernstein inequalities, we arrive at
\begin{eqnarray*}
\|\dot{\Delta}_{j}R(a,b)\|_{L^q}
&\lesssim&2^{(d-\frac{d}{q})j}\|\dot{\Delta}_{j}\Big(\sum_{j'}\tilde{\dot{\Delta}}_{j'}a\dot{\Delta}_{j'}b\Big)\|_{L^1}\\
&\lesssim&2^{(d-\frac{d}{q})j}\sum_{j'\geq j-2}\|\dot{\Delta}_{j'}a\|_{L^{2}}\|\dot{\Delta}_{j'}b\|_{L^2}\\
&\lesssim&2^{(d-\frac{d}{q})j}\sum_{j'\geq j-2}2^{(k-s+\frac{d}{p}-d)j'}2^{(\frac{d}{p}-k_{1})j'}\|\dot{\Delta}_{j'}a\|_{L^{p}}2^{(s-k_{2})j'}\|\dot{\Delta}_{j'}b\|_{L^p}\\
&\lesssim&2^{(d-\frac{d}{q})j}\sum_{j'\geq j-2}2^{(k-s+\frac{d}{p}-d)j'}\|a\|_{\dot{B}^{\frac{d}{p}-k_{1}}_{p,\infty}}\|b\|_{\dot{B}^{s-k_{2}}_{p,\infty}},
\end{eqnarray*}
where H$\ddot{o}$lder inequality for series was performed in the last inequality. If $s>k-\frac{d}{p'}$, then $s-k-\frac{d}{p}+d>0$, so one can immediately have
$$2^{(s-k+d-\frac{d}{p})j}\sum_{j'\geq j-2}2^{(k-s+\frac{d}{p}-d)j'}\leq\sum_{j'\geq j-2}2^{(k-s+\frac{d}{p}-d)(j'-j)}\leq c$$
for some constant $c$ depends on $d,p,k,s$. Therefore, we conclude with
\begin{eqnarray*}
\|R(a,b)\|_{\dot{B}^{s-k+\frac{d}{q}-\frac{d}{p}}_{q,\infty}}=\sup_{j\in\mathbb{R}}2^{(s-k+\frac{d}{q}-\frac{d}{p})j}\|\dot{\Delta}_{j}R(a,b)\|_{L^q}
\lesssim\|a\|_{\dot{B}^{\frac{d}{p}-k_{1}}_{p,\infty}}\|b\|_{\dot{B}^{s-k_{2}}_{p,\infty}}.
\end{eqnarray*}

On the other hand, we deal with the case $2< p\leq2q$. By H\"{o}lder and Bernstein inequalities, we have
\begin{eqnarray*}
\|\dot{\Delta}_{j}R(a,b)\|_{L^q}
&\lesssim&2^{(\frac{2d}{p}-\frac{d}{q})j}\|\dot{\Delta}_{j}\Big(\sum_{j'}\tilde{\dot{\Delta}}_{j'}a\dot{\Delta}_{j'}b\Big)\|_{L^\frac{p}{2}}\\
&\lesssim&2^{(\frac{2d}{p}-\frac{d}{q})j}\sum_{j'\geq j-2}\|\dot{\Delta}_{j'}a\|_{L^{p}}\|\dot{\Delta}_{j'}b\|_{L^p}\\
&\lesssim&2^{(\frac{2d}{p}-\frac{d}{q})j}\sum_{j'\geq j-2}2^{(k-s-\frac{d}{p})j'}2^{(\frac{d}{p}-k_{1})j'}\|\dot{\Delta}_{j'}a\|_{L^{p}}2^{(s-k_{2})j'}\|\dot{\Delta}_{j'}b\|_{L^p}\\
&\lesssim&2^{(\frac{2d}{p}-\frac{d}{q})j}\sum_{j'\geq j-2}2^{(k-s-\frac{d}{p})j'}\|a\|_{\dot{B}^{\frac{d}{p}-k_{1}}_{p,\infty}}\|b\|_{\dot{B}^{s-k_{2}}_{p,\infty}}.
\end{eqnarray*}
Since $s>k-\frac{d}{p}$, i.e. $s-k+\frac{d}{p}>0$, we similarly get
\begin{eqnarray*}
\|R(a,b)\|_{\dot{B}^{s-k+\frac{d}{q}-\frac{d}{p}}_{q,\infty}}
\lesssim\|a\|_{\dot{B}^{\frac{d}{p}-k_{1}}_{p,\infty}}\|b\|_{\dot{B}^{s-k_{2}}_{p,\infty}}.
\end{eqnarray*}
\end{proof}

In order to investigate the Gevrey regularity of solutions in the $L^p$ framework, the following inequalities
involving Gevrey multiplier are needed.
\begin{lem}\label{lem key G}
Let $s,k,k_{1},k_{2}\in \mathbb{R}$ and $1 < p,q<\infty$. Then there holds for $k=k_{1}+k_{2}$ and $1=1/p+1/p'$ such that
\begin{itemize}
\item [$(i)$] $\|e^{\sqrt {c_{0}t}\Lambda_1}T_{a}b\|_{\dot{B}^{s-k+\frac{d}{q}-\frac{d}{p}}_{q,\infty}}\lesssim\|A\|_{\dot{B}^{\frac{d}{p}-k_{1}}_{p,\infty}}\|B\|_{\dot{B}^{s-k_{2}}_{p,\infty}} \,\mathrm{if}\,
    k_{1}> d\max(0,\frac{1}{q}-\frac{1}{p})\,  \mbox{and} \, \, p\leq2q;$

\item [$(ii)$] $\|e^{\sqrt {c_{0}t}\Lambda_1}R(a,b)\|_{\dot{B}^{s-k+\frac{d}{q}-\frac{d}{p}}_{q,\infty}}\lesssim\|A\|_{\dot{B}^{\frac{d}{p}-k_{1}}_{p,\infty}}\|B\|_{\dot{B}^{s-k_{2}}_{p,\infty}}\,\mathrm{if}\,
s>k-d\min(\frac{1}{p},\frac{1}{p'})\, \mbox{and} \, \,  p\leq2q.$
\end{itemize}
where $A\triangleq e^{\sqrt {c_{0}t}\Lambda_1}a$ and $B\triangleq e^{\sqrt {c_{0}t}\Lambda_1}b$.
\end{lem}
\begin{proof}
As a matter of fact, it suffices to give the paraproduct estimate $(i)$ with $q\leq p$, since the proof of Lemma \ref{lem key G} is similar to Lemma \ref{lem key}. By the definition of the paraproduct and of $\mathcal{B}_{t}$, we have
\begin{equation}\label{estimlocT}
e^{\sqrt {c_{0}t}\Lambda_1} T_{a}b=\sum_{j\in \mathbb{Z}}W_{j}\,\,\, \mbox{with}\,\,\,
W_{j}\triangleq \mathcal{B}_t(\dot S_{j-1}A,\dot\Delta_jB).
\end{equation}
Thanks to Lemma \ref{lem5.3}, we obtain
\begin{eqnarray*}
\|W_{j}\|_{L^q}\lesssim \|\dot S_{j-1}A\|_{L^m}\|\dot\Delta_jB\|_{L^p}\lesssim \sum_{j'\leq j-2}\|\dot\Delta_{j'}A\|_{L^m}\|\dot\Delta_jB\|_{L^p},
\end{eqnarray*}
where $\frac{1}{q}=\frac{1}{m}+\frac{1}{p}$ is used. Since $p\leq2q$ implies that $p\leq m$, it follows from Bernstein inequality that
\begin{eqnarray*}
\|W_{j}\|_{L^q}\lesssim \sum_{j'\leq j-2}2^{(\frac{d}{p}-\frac{d}{q}+k_{1})j'}2^{(\frac{d}{p}-k_{1})j'}\|\dot{\Delta}_{j'}A\|_{L^p}\|\dot{\Delta}_{j}B\|_{L^p}.
\end{eqnarray*}
As $\frac{d}{p}-\frac{d}{q}+k_{1}> 0$, H\"{o}lder inequality enables us to obtain
\begin{eqnarray*}
\sup_{j\in \mathbb{Z}}2^{j(s-k+\frac d q -\frac d p)}\|W_{j}\|_{L^q}
\lesssim\|A\|_{\dot{B}^{\frac{d}{p}-k_{1}}_{p,\infty}} \|B\|_{\dot{B}^{s-k_{2}}_{p,\infty}}
\end{eqnarray*}
and one may conclude to $(i)$ with $q\leq p$ by using Proposition \ref{prop2.1}.
\end{proof}

\begin{prop} \label{prop4.1}
Let $F$ be a real analytic function in a neighborhood of 0, such that
$F(0) = 0$. Let $1< p\leq 2d$ and $s>\frac{d}{p}$. There exists a constant $\tilde{R}_{0}$ such that if for some $T>0$
$$\|e^{\sqrt {c_{0}t}\Lambda_1}u\|_{\tilde{L}_{T}^{\infty}(\dot{B}^{\frac{d}{p}}_{p,1})}\leq \tilde{R}_{0},$$
then
\begin{eqnarray}\label{composite}
\|e^{\sqrt {c_{0}t}\Lambda_1}F(u)\|_{\tilde{L}_{T}^{r}(\dot{B}^{s}_{p,1})}\leq \bar{C}_{\tilde{R}_{0}}\|e^{\sqrt {c_{0}t}\Lambda_1}u\|_{\tilde{L}_{T}^{r}(\dot{B}^{s}_{p,1})}\end{eqnarray}
for $1\leq r\leq\infty$, where $\bar{C}_{\tilde{R}_{0}}$ depends only on $d,p,F$ and $\tilde{R}_{0}$.
\end{prop}

\begin{proof}
Proposition \ref{prop4.1} generalizes Proposition \ref{prop5.2} such that the regularity index could be larger than $d/p$. Hence, Propositions
\ref{prop5.2} and \ref{prop4.1} are the analogue of Proposition \ref{prop2.25}, where the Gevrey regularity is involved in the estimate of composite functions.

For convenience, one can write $s=\frac{d}{p}+k$
where $k>0$. Set $\iota_k=k$ if $k\in \mathbb{Z}$, otherwise $\iota_k=[k+1]$. Since $F(0) = 0$, by utilizing equivalent norm and Taylor extension, we have
\begin{eqnarray*}
\|e^{\sqrt {c_{0}t}\Lambda_1}F(z)\|_{\tilde{L}_{T}^{r}(\dot{B}^{\frac{d}{p}+k}_{p,1})}
\lesssim\sup_{|\alpha|=\iota_k}\|\sum_{n=1}^{\infty}|a_{n}|e^{\sqrt {c_{0}t}\Lambda_1}\partial^{\alpha}(z^{n})\|_{\tilde{L}_{T}^{r}(\dot{B}^{\frac{d}{p}+k-\iota_k}_{p,1})}.
\end{eqnarray*}
Without loss of generality, we focus on the case $k\in(0, 1]$. Hence, we have $\iota_k=1$ and
\begin{eqnarray*}
&&\|e^{\sqrt {c_{0}t}\Lambda_1}F(z)\|_{\tilde{L}_{T}^{r}(\dot{B}^{\frac{d}{p}+k}_{p,1})}\\
&\lesssim&\sup_{|\alpha|=1}\sum_{n=1}^{\infty}\sum_{j=1}^{d}|a_{n}|n\|e^{\sqrt {c_{0}t}\Lambda_1}z^{n-1}\partial_{x_{j}}z\|_{\tilde{L}_{T}^{r}(\dot{B}^{\frac{d}{p}+k-1}_{p,1})}.
\end{eqnarray*}
Owing to $\frac{d}{p}-1< \frac{d}{p}+k-1\leq\frac{d}{p}$, it follows from Proposition \ref{prop5.1} that
\begin{eqnarray}\label{jjkk}
&&\|e^{\sqrt {c_{0}t}\Lambda_1}F(z)\|_{\tilde{L}_{T}^{r}(\dot{B}^{\frac{d}{p}+k}_{p,1})}\\
\nonumber&\lesssim&\sup_{|\alpha|=1}\sum_{n=1}^{\infty}\sum_{j=1}^{d}n|a_{n}|\|e^{\sqrt
{c_{0}t}\Lambda_1}z\|^{n-1}_{\tilde{L}_{T}^{\infty}(\dot{B}^{\frac{d}{p}}_{p,1})}\|e^{\sqrt
{c_{0}t}\Lambda_1}\partial_{x_{j}}z\|_{\tilde{L}_{T}^{r}(\dot{B}^{\frac{d}{p}+k-1}_{p,1})}\\
\nonumber&\leq& C\sum_{n=1}^{\infty}n|a_{n}|\|e^{\sqrt
{c_{0}t}\Lambda_1}z\|^{n-1}_{\tilde{L}_{T}^{\infty}(\dot{B}^{\frac{d}{p}}_{p,1})}\|e^{\sqrt
{c_{0}t}\Lambda_1}z\|_{\tilde{L}_{T}^{r}(\dot{B}^{\frac{d}{p}+k}_{p,1})}
\\
\nonumber&\leq& \mathcal{F}(C\|e^{\sqrt
{c_{0}t}\Lambda_1}z\|^{n-1}_{\tilde{L}_{T}^{\infty}(\dot{B}^{\frac{d}{p}}_{p,1})})\|e^{\sqrt
{c_{0}t}\Lambda_1}z\|_{\tilde{L}_{T}^{r}(\dot{B}^{\frac{d}{p}+k}_{p,1})},\end{eqnarray}
provided that $1< p\leq 2d$. Clearly, $\sum n|a_{n}|z^{n-1}$ is the differentiation term-by-term of the original series
$F(z)=\sum |a_{n}|z^{n}$, so its radius of convergence won't change. Let $R_{F}>0$ the convergence radius of the series.
So when $\|e^{\sqrt {c_{0}t}\Lambda_1}z\|_{\tilde{L}^{\infty}(\dot{B}^{\frac{d}{p}}_{p,1})}\leq \frac{R_{F}}{2C}\triangleq \tilde{R}_{0}$, we have
\begin{eqnarray*}
\|e^{\sqrt {c_{0}t}\Lambda_1}F(z)\|_{\tilde{L}_{T}^{r}(\dot{B}^{\frac{d}{p}+k}_{p,1})}\leq \bar{C}_{\tilde{R}_{0}}
\|e^{\sqrt {c_{0}t}\Lambda_1}z\|_{\tilde{L}_{T}^{r}(\dot{B}^{\frac{d}{p}+k}_{p,1})}
\end{eqnarray*}
with $\bar{C}_{\tilde{R}_{0}}=\sup _{\mathcal{B}(0,\frac{R_{F}}{2})}|\mathcal{F}(z)|$. Thus we can continue the similar process for $k\in(1,2],(2,3],...,$
and finally achieve \eqref{composite} (by applying Leibniz's formula and interpolation if necessary). The details are left to the interested reader.
\end{proof}

\noindent {\bf Acknowledgments:}\ \
The second author (J. Xu) is partially supported by the National Natural Science Foundation of China (11871274, 12031006).


\begin{thebibliography}{00}
\bibitem{BBT}
H. Bae, A. Biswas and E. Tadmor, Analyticity and decay estimates of the Navier-Stokes equations in critical Besov spaces, \textit{Arch. Ration. Mech. Anal.}, {\bf{205}}, 963-991 (2012).

\bibitem{BCD}
H. Bahouri, J.-Y. Chemin and R. Danchin, {\it Fourier Analysis and Nonlinear Partial Differential Equations,} Grundlehren der mathematischen Wissenschaften, {\bf 343}, Springer (2011).

\bibitem{BDDJ}
S. Benzoni-Gavage, R. Danchin, S. Descombes and D. Jamet, Structure of Korteweg models and stability of diffuse interfaces, {\it Interfaces and Free Boundaries}, {\bf 7}, 371-414 (2005).

\bibitem{C}
M. Cannone, A generalization of a theorem by Kato on Navier-Stokes equations, \textit{Rev. Mat. Iberoamericana}, {\bf 13}, 515-542 (1997).

\bibitem{CP}
M.~Cannone and F.~Planchon, Self-similar solutions for Navier-Stokes equations in
$\mathbb{R}^3$,  \textit{Comm. Part. Differ. Equs.}, {\bf 21}, 179-193 (1996).


\bibitem{CJY}
J.~Y. Chemin: Th\'{e}or\`{e}mes d'unicit\'{e} pour le syst\`{e}%
m de Navier-Stokes tridimensionnel, \textit{J. Anal. Math.}, {\bf 77}, 27-50 (1999).
\bibitem{CD}
F. Charve and R. Danchin, A global existence result for the compressible Navier-Stokes equations in the critical $L^p$ framework. \textit{Arch. Ration. Mech. Anal.}, {\bf{198}}, 233-271 (2010).

\bibitem{CDX}
F. Charve,  R. Danchin and J. Xu, Gevrey analyticity and decay for the compressible Navier-Stokes system with capillarity, {\it  Indiana Univ. Math. J.}, {\bf 5}, 1903-1944  (2021).



\bibitem{CK}
N. Chikami and T. Kobayashi, Global well-posedness and time-decay estimates of the compressible Navier-Stokes-Korteweg system in critical Besov spaces, {\it J. Math. Fluid Mech.}, {\bf{21}}: Art.31 (2019).

\bibitem{CL}
J. Y. Chemin and N. Lerner, Flot de champs de vecteurs non lipschitziens et $\acute{e}$quations de Navier-Stokes, \textit{J. Differential Equations}, {\bf{121}}, 314-328 (1995).

\bibitem{CMZ}
Q. Chen, C. Miao and Z. Zhang, Global well-posedness for compressible Navier-Stokes equations with highly oscillating initial velocity, \textit{Comm. Pure Appl. Math.}, {\bf{63}}, 1173-1224 (2010).

\bibitem{D}
R. Danchin, Global existence in critical spaces for compressible Navier-Stokes equations. \textit{Invent. Math}, {\bf{141}}, 579-614 (2000).

\bibitem{D2} R. Danchin, {\em Fourier Analysis Methods for the Compressible Navier-Stokes Equations,}
Handbook of Mathematical Analysis in Mechanics of Viscous Fluids, Y. Giga and A. Novotny editors,
Springer International Publishing Switzerland, 2016.

\bibitem{DD}
R. Danchin and B. Desjardins, Existence of solutions for compressible fluid models of Korteweg type, {\it Ann. Inst. H. Poincar\'{e} Anal. Non
Lin\'{e}aire}, {\bf 18}, 97-133 (2001).

\bibitem{DH}
R. Danchin and L. He, The incompressible limit in $L^p$ type critical spaces, { \it Math. Ann.}, {\bf 366}, 1365-1402 (2016).

\bibitem{DX}
R. Danchin and J. Xu, Optimal time-decay estimates for the compressible Navier-Stokes equations in the critical $L^p$ framework, \textit{Arch. Ration. Mech. Anal}, {\bf{224}}, 53-90 (2017).


 \bibitem{DDL}
D. Bresch, B. Desjardins and C. K. Lin, On some compressible fluid models: Korteweg, lubrication,
 and shallow water systems, \textit{Comm. Part. Differ. Equs.}, {\bf 28}, 843-868 (2003).

\bibitem{DS}
J.~E. Dunn and J. Serrin, On the thermomechanics of interstitial working, {\it Arch. Ration. Mech. Anal}, {\bf 88}, 95-133 (1985).

\bibitem{FT1}
C. Foias and R. Temam, Some analytic and geometric properties of the solutions of the
 Navier-Stokes equations, \textit{J. Math. Pures Appl.}, {\bf 58}, 339-368 (1979).


\bibitem{FT2}
C. Foias and R. Temam, Gevrey class regularity for the solutions of the Navier-Stokes
equations, \textit{J. Funct. Anal.} {\bf 87}, 359-369 (1989).

\bibitem{FK}
H. Fujita and T. Kato, On the Navier-Stokes initial value problem, \textit{Arch. Ration. Mech. Anal.},  {\bf{16}}, 269-315 (1964).

\bibitem{H2}
B. Haspot, Existence of global strong solutions in critical spaces for barotropic viscous fluids,
{\it Arch. Ration. Mech. Anal.}, {\bf 202}, 427-460 (2011).

\bibitem{H3}
B. Haspot, Global strong solution for the Korteweg system with quantum pressure in dimension $N\geq2$, {\it Math. Ann}, {\bf 367}, 667-700 (2017).



\bibitem{HL}
H. Hattori and D. Li, Global solutions of a high-dimensional system for Korteweg materials, {\it J. Math. Anal. Appl}, {\bf 198}, 84-97 (1996).

\bibitem{HL1}
H. Hattori and  D. Li, Solutions for two-dimensional system for materials of Korteweg type, {\it SIAM J. Math. Anal}, {\bf 25},
85-98 (1994).

\bibitem{HO}
D. Hoff, Global solutions of the Navier-Stokes equations for multidimensional compressible flow with
 discontinuous initial data, \textit{J. Differ. Equ.}, {\bf 120}, 215-254 (1995).

\bibitem{HZ} D. Hoff and K. Zumbrun:  Multidimensional diffusion waves for the
Navier-Stokes equations of compressible flow, \textit{Indiana Univ. Math. J.}, {\bf{44}}, 604--676 (1995).

\bibitem{HHS}
F. Huang, H. Hong and X. Shi, Existence of smooth solutions for the compressible baratropic Navier-Stokes-Korteweg system without incresing pressure law, {\it Math. Meth. Appl. Sci.}, {\bf 43}, 5073-5096 (2020).

\bibitem{L-cras} P.-G. Lemari\'e-Rieusset, Une remarque sur l'analycit\'e des solutions milds des \'equations
de Navier-Stokes dans $\R^3,$  {\em C. R. Acad. Sci. Paris, S\'erie 1}, {\bf 330}, 183--186 (2000).

\bibitem{L-book} P.-G. Lemari\'e-Rieusset, {\it Recent Developments in the Navier-Stokes Problem}, Chapman
\& Hall/CRC Research Notes in Mathematics, vol. 431. Chapman \& Hall/CRC,
Boca Raton, 2002.

\bibitem{KSX1}
S. Kawashima, Y. Shibata and J. Xu, The $L^p$ energy methods and decay for the compressible Navier-Stokes equations with capillarity,
 {\it J. Math. Pure Anal.} {\bf 154},  146-184 (2021).


\bibitem{KSX2}
S. Kawashima, Y. Shibata and J. Xu, Dissipative structure for symmetric hyperbolic-parabolic systems with Korteweg-type dispersion,
 {\it Comm. Part. Differ. Equs.} {\bf 47}, 378-400 (2021).

\bibitem{KM}
T. Kobayashi and M. Murata, The global well-posedness of the compressible fluid model of Korteweg type for the critical case. {\it Differential Integral Equations} {\bf 34}, 245-264 (2021).

 \bibitem{KT}
T. Kobayashi and K. Tsuda, Global existence and time decay estimate of solutions to the compressible Navier-Stokes-Korteweg system under critical condition. {\it Asymptot. Anal.} {\bf 121},  195-217 (2021).

\bibitem{K}
D.~J. Korteweg, Sur la forme que prennent les $\acute{e}$quations du mouvement des fluides si l$'$on tient compte des forces capillaires
par des variations de densit$\acute{e}$, {\it Arch.  Néer. Sci. Exactes Sér. II} {\bf 6}, 1-24 (1901).

\bibitem{K2}
M. Kotschote, Strong solutions for a compressible fluid model of Korteweg type, {\it Ann. Inst. H. Poincar\'{e} Anal. Non
Lin\'{e}aire}, {\bf 25}, 679-696 (2008).


\bibitem{M}
K.  Masuda, On the analyticity and the unique continuation theorem for Navier-Stokes
 equations, \textit{Proc. Jpn. Acad. Ser. A Math. Sci.}, {\bf{43}}, 827-832 (1967).

\bibitem{MS}
 M. Murata and Y. Shibata, The global well-posedness for the compressible fluid model of Korteweg type, {\it SIAM J. Math. Anal.}, {\bf 52}, 6313-6337 (2020).


\bibitem{OT}
M. Oliver and E.~Titi, Remark on the rate of decay of higher order derivatives for
solutions to the Navier-Stokes equations in $\mathbb{R}^{n}$,  {\it J. Funct. Anal.}, {\bf 172}, 1-18 (2000).

\bibitem{TG}
T. Tang and H.~J.~Gao, On the compressible Navier-Stokes-Korteweg equations, {\it Discrete Contin. Dyn. Syst.-B}, {\bf 21}, 2745-2766 (2016).

\bibitem{V}
J.~F. Van der Waals, Thermodynamische Theorie der Kapillarit$\ddot{a}$t unter Voraussetzung stetiger Dichte$\ddot{a}$nderung, {\it Phys.
Chem}. {\bf 13}, 657-725 (1894).

\bibitem{W2}
K. Watanabe, Global large solutions and incompressible limit for the compressible Navier-Stokes system with capillarity. {\it J. Math. Anal. Appl.} {\bf 518}, 26 pp (2023).

\bibitem{W}
K. Watanabe, Global existence of the Navier-Stokes-Korteweg equations with a
non-decreasing pressure in $L^p$-framework,  arXiv:1907.07752v2 (2019).

\bibitem{XX}
Z. Xin and J. Xu, Optimal decay for the compressible Navier-Stokes equations without additional smallness assumptions, \textit{J. Differential Equations}, {\bf{274}}, 543-575 (2021).

\bibitem{X}
J. Xu, A low-frequency assumption for optimal time-decay estimates to the compressible Navier-Stokes equations, \textit{Comm. Math. Phys}, {\bf{371}}, 525-560 (2019).

\bibitem{YYW}
Y. Yu; X. Yang and X Wu, Global smooth solutions of 3-D Navier-Stokes-Korteweg equations with large initial data. {\it Math. Methods Appl. Sci.} {\bf 45}  6165-6180 (2022).

\bibitem{ZL}
X. Zhai and Y. Li, Global large solutions and optimal time-decay estimates to the Korteweg system. {\it Discrete Contin. Dyn. Syst.} {\bf 41},  1387-1413 (2021).

\bibitem{Zhang}
S. Zhang,  A class of global large solutions to the compressible Navier-Stokes-Korteweg system in critical Besov spaces. {\it J. Evol. Equ.} {\bf 20}, 1531-1561 (2020).

\end{thebibliography}
\end{document}